\documentclass[11pt]{article}
\usepackage{enumerate}
\usepackage{amssymb,a4wide,latexsym,makeidx,epsfig}
\usepackage{amsthm}
\usepackage{dsfont}
\usepackage{amsmath}
\usepackage{lipsum}
\usepackage{enumerate}
\usepackage{mathrsfs}
\usepackage{xcolor}
\usepackage{tikz}
\usepackage{setspace}
\usepackage{geometry}
\usepackage{appendix}
\usepackage{multirow}
\usepackage{subcaption}
\usepackage[colorlinks=true, linkcolor=black, anchorcolor=black, citecolor=black, urlcolor=black, CJKbookmarks=true]{hyperref}
\usepackage{microtype}
\usepackage{colonequals}

\allowdisplaybreaks

\newtheorem{theorem}{Theorem}[section]
\newtheorem{remark}{Remark}[section]

\newtheorem{definition}[theorem]{Definition}
\newtheorem{lemma}[theorem]{Lemma}

\newtheorem{observation}[theorem]{Observation}
\newtheorem{proposition}[theorem]{Proposition}

\newtheorem{problem}[theorem]{Problem}

\newtheorem{claim}{Claim}[section]

\numberwithin{equation}{section}

\begin{document}
\textwidth 150mm \textheight 225mm

\title{Edge-colored 3-uniform hypergraphs without rainbow paths of length 3 and applications to Ramsey theory
}

\author{
Xihe Li\footnote{School of Mathematics and Statistics, Shaanxi Normal University, Xi'an, Shaanxi 710119, China.}~\footnote{Corresponding author.}~~~~~~
Runshan Wang\footnotemark[1]~~~~~~
}
\date{}
\maketitle
\newcommand\blfootnote[1]{%
\begingroup
\renewcommand\thefootnote{}\footnote{#1}%
\addtocounter{footnote}{-1}%
\endgroup
}
\blfootnote{E-mail addresses: xiheli@snnu.edu.cn, rswangmath@163.com.}
\begin{center}
\begin{minipage}{120mm}
\begin{center}
{\small {\bf Abstract}}
\end{center}
{\small
\hspace{2em}Motivated by problems in Ramsey theory, we study edge-colorings of 3-uniform hypergraphs that contain no rainbow paths of length 3.
We consider the following three natural 3-uniform paths of length 3: the tight path $\mathcal{T}=\{v_1v_2v_3, v_2v_3v_4, v_3v_4v_5\}$, the messy path $\mathcal{M}=\{v_1v_2v_3, v_2v_3v_4, v_4v_5v_6\}$ and the loose path $\mathcal{L}=\{v_1v_2v_3,$ $v_3v_4v_5, v_5v_6v_7\}$.
In this paper, we characterize the structures of rainbow $\mathcal{T}$-free, rainbow $\mathcal{M}$-free and rainbow $\mathcal{L}$-free edge-colorings of complete 3-uniform hypergraphs $K_n^{(3)}$, respectively.
This extends a result of Thomason and Wagner (2007) on edge-colored complete graphs $K_n$ without rainbow paths of length 3.

\hspace{2em}As applications, we obtain several Ramsey-type results.
Given two $3$-uniform hypergraphs $H$ and $G$, the {\it constrained Ramsey number} $f(H,G)$ is defined as the minimum integer $n$ such that in every edge-coloring of $K^{(3)}_n$ with any number of colors, there is either a monochromatic copy of $H$ or a rainbow copy of $G$.
For $G\in \{\mathcal{T}, \mathcal{M}, \mathcal{L}\}$ and infinitely many 3-uniform hypergraphs $H$, we show that $f(H, G)=R_2(H)$, where $R_2(H)$ is the 2-colored Ramsey number of $H$.
Given a $3$-uniform hypergraph $G$ and an integer $n\geq |V(G)|$, the {\it anti-Ramsey number} $ar(n, G)$ is the minimum integer $k$ such that in every edge-coloring of $K^{(3)}_n$ with at least $k$ colors, there is a rainbow copy of $G$.
We show that $ar(n, \mathcal{T})=\left\lfloor\frac{n}{3}\right\rfloor+2$ for $n\geq 5$,
$ar(n, \mathcal{M})=3$ for $n\geq 7$,
and $ar(n, \mathcal{L})=n$ for $n\geq 7$.
Our Ramsey-type results extend results of Gy\'{a}rf\'{a}s, Lehel and Schelp (2007) and of Liu (2024) on constrained Ramsey numbers, and improve a result of Tang, Li and Yan (2022) on anti-Ramsey numbers.
We also obtain multipartite analogues of our structural and Ramsey-type results.

\vskip 0.1in \noindent {\bf AMS Subject Classification (2020)}: \ 05C55, 05D10, 05C35
\vskip 0.1in \noindent {\bf Keywords}: \ rainbow subgraph, hypergraph paths, constrained Ramsey number, rainbow Ramsey number, anti-Ramsey number
}
\end{minipage}
\end{center}

\section{Introduction}
\label{sec:introduction}

Structural partition results play an important role in graph theory.
A typical example is Szemer\'{e}di's Regularity Lemma~\cite{Sze}, which informally states that the vertex set of every graph can be partitioned into a bounded number of parts, with pseudorandom edges between almost all pairs of the parts.
The regularity lemma has wide applications in graph theory, additive combinatorics, discrete geometry and theoretical computer science.
Another example is Gallai's Partition Theorem~\cite{Gallai}, which states that for every edge-colored complete graph without a rainbow $K_3$, the vertex set can be partitioned into at least two parts such that there is only one color on edges between each pair of distinct parts, and there are at most two colors on edges between all the parts.
Here an edge-colored graph is called {\it rainbow} if all edges are colored differently.
Gallai's Partition Theorem has many applications to Ramsey theory~\cite{FoGP,LiBW,LMSSS}, information theory~\cite{KoSi}, the study of perfect graphs~\cite{CaEL,KoST}, partially ordered sets~\cite{Gallai} and relation algebras~\cite{KrMa}.

Motivated by problems in Ramsey theory, Thomason and Wagner~\cite{ThWa} characterized the structures of edge-colored complete graphs with no short rainbow paths.
Before presenting the result of Thomason and Wagner, we first introduce some necessary notation.
For a positive integer $k$, let $[k]\colonequals \{1, 2, \ldots, k\}$.
Given a (hyper)graph $G$, we refer to a mapping $c: E(G) \to [k]$ as a {\it $k$-edge-coloring} (not necessarily a proper edge-coloring) of $G$.
Note that we do not require all the $k$ colors to be used in a $k$-edge-coloring, i.e., the mapping need not be surjective.
For an edge $e\in E(G)$, let $c(e)$ be the color assigned to $e$, and let $C(G)\colonequals \{c(e)\colon\, e\in E(G)\}$.
Throughout this paper, we assume $C(G)=\{1, 2, \ldots, |C(G)|\}$.
Note that $G$ is rainbow if and only if $|C(G)|=|E(G)|$.
If $|C(G)|=1$, then $G$ is called {\it monochromatic}.
An edge-colored (hyper)graph $G$ is called {\it rainbow $H$-free} if every copy of $H$ in $G$ receives at most $|E(H)|-1$ colors.
For any subset $U\subseteq V(G)$, let $C(U)\colonequals \{c(e)\colon\, e\in E(G), e\subseteq U\}$,
$G[U]$ be the subgraph of $G$ induced by $U$,
and $G-U$ be the subgraph of $G$ induced by $V(G)\setminus U$.
If $U$ consists of a single vertex $u$, then we simply write $G-\{u\}$ as $G-u$.
In 2007, Thomason and Wagner~\cite{ThWa} obtained the following result.\footnote{The initial assertions made by Thomason and Wagner are formulated somewhat differently (see \cite[Theorems~1 and 2]{ThWa}). Theorem~\ref{th:ThWa} serves as a convenient reformulation of their results.}

\begin{theorem}{\normalfont (\cite{ThWa})}\label{th:ThWa}
For any integer $n\geq 5$, let $G$ be an edge-colored complete graph $K_n$.
\begin{itemize}
\item[{\rm (i)}] If $G$ contains no rainbow path of length 3, then $|C(G)|\leq 2$.

\item[{\rm (ii)}] If $G$ contains no rainbow path of length 4, then either $|C(G)|\leq 3$ or at least one of the following two statements holds $($after renumbering the colors if necessary$)$:
\begin{itemize}
\item[{\rm (a)}] there exists a subset $U\subseteq V(G)$ with $|U|\leq 3$ such that $G-U$ is monochromatic;

\item[{\rm (b)}] we can partition $V(G)$ into $|C(G)|-1$ parts $V_2, V_3, \ldots, V_{|C(G)|}$ satisfying $\{i\}\subseteq C(V_i)\subseteq \{1,i\}$ for every $i\in \{2, 3, \ldots, |C(G)|\}$, and every edge joining distinct parts has color 1.\footnote{We remark that there is no $V_1$ in this partition.}
\end{itemize}
\end{itemize}
\end{theorem}

In this paper, we consider hypergraph analogues of this problem.
A {\it hypergraph} $H$ is a pair $(V(H), E(H))$, where $V(H)$ is the vertex set and $E(H)\subseteq \{e\colon\, e\subseteq V(H)\}$ is the edge set.
If every edge $e\in E(H)$ satisfies $|e|=r$ ($r\geq 2$), then $H$ is called an {\it $r$-uniform hypergraph} (or {\it $r$-graph} for short).
In particular, every ordinary graph is a 2-graph.
The {\it complete $r$-uniform hypergraph} $K^{(r)}_n$ is an $r$-graph on $n$ vertices in which every $r$ vertices form an edge.
The $r$-uniform {\it tight path} $\mathbb{P}^{(r)}_{\ell}$ of length $\ell$ is an $r$-graph with vertex set $\{v_1, v_2, \ldots, v_{\ell+ r-1}\}$ and edge set $\{e_1, e_2, \ldots, e_{\ell}\}$ such that $e_i=\{v_{i+j}\colon\, 0\leq j\leq r-1\}$ for each $i\in [\ell]$.
The $r$-uniform {\it loose path} (sometimes called {\it linear path}) $P^{(r)}_{\ell}$ of length $\ell$ is an $r$-graph with vertex set $\{v_1, v_2, \ldots, v_{\ell(r-1)+1}\}$ and edge set $\{e_1, e_2, \ldots, e_{\ell}\}$ such that $e_i=\{v_{(i-1)(r-1)+j}\colon\, 1\leq j\leq r\}$ for each $i\in [\ell]$.
We will mainly consider 3-uniform hypergraphs.
Specifically, we study the following three natural 3-uniform paths of length \(3\): the tight path $\mathcal{T}=\{v_1v_2v_3, v_2v_3v_4, v_3v_4v_5\}$, the messy path $\mathcal{M}=\{v_1v_2v_3, v_2v_3v_4, v_4v_5v_6\}$ and the loose path $\mathcal{L}=\{v_1v_2v_3, v_3v_4v_5, v_5v_6v_7\}$ (see Figure~\ref{fig:3-paths} for an illustration).
Note that $\mathcal{T}=\mathbb{P}^{(3)}_3$ and $\mathcal{L}=P_3^{(3)}$.

\begin{figure}[htbp]
\begin{center}
\begin{tikzpicture}[scale=0.06,auto,swap]
\tikzstyle{vertex}=[circle,draw=black,fill=black]
\draw (-77,-20) node {$\mathcal{T}$};
\node[vertex,scale=0.5] (Tv1) at (-99,0) {}; \draw (-99,7) node {$v_1$};
\node[vertex,scale=0.5] (Tv2) at (-88,0) {}; \draw (-88,7) node {$v_2$};
\node[vertex,scale=0.5] (Tv3) at (-77,0) {}; \draw (-77,7) node {$v_3$};
\node[vertex,scale=0.5] (Tv4) at (-66,0) {}; \draw (-66,7) node {$v_4$};
\node[vertex,scale=0.5] (Tv5) at (-55,0) {}; \draw (-55,7) node {$v_5$};
\draw[rotate=0](-66,0)ellipse(15 and 3.5); \draw (-64,-8) node {$e_3$};
\draw[rotate=0](-88,0)ellipse(15 and 3.5); \draw (-90,-8) node {$e_1$};
\draw[rotate=0](-77,0)ellipse(15 and 3.5); \draw (-77,-8) node {$e_2$};

\draw (-2.5,-20) node {$\mathcal{M}$};
\node[vertex,scale=0.5] (Mv1) at (-30,0) {}; \draw (-30,7) node {$v_1$};
\node[vertex,scale=0.5] (Mv2) at (-19,0) {}; \draw (-19,7) node {$v_2$};
\node[vertex,scale=0.5] (Mv3) at (-8,0) {}; \draw (-8,7) node {$v_3$};
\node[vertex,scale=0.5] (Mv4) at (3,0) {}; \draw (3,7) node {$v_4$};
\node[vertex,scale=0.5] (Mv5) at (14,0) {}; \draw (14,7) node {$v_5$};
\node[vertex,scale=0.5] (Mv6) at (25,0) {}; \draw (25,7) node {$v_6$};
\draw[rotate=0](14,0)ellipse(15 and 3.5); \draw (14,-8) node {$e_3$};
\draw[rotate=0](-19,0)ellipse(15 and 3.5); \draw (-20,-8) node {$e_1$};
\draw[rotate=0](-8,0)ellipse(15 and 3.5); \draw (-6,-8) node {$e_2$};

\draw (83,-20) node {$\mathcal{L}$};
\node[vertex,scale=0.5] (Lv1) at (50,0) {}; \draw (50,7) node {$v_1$};
\node[vertex,scale=0.5] (Lv2) at (61,0) {}; \draw (61,7) node {$v_2$};
\node[vertex,scale=0.5] (Lv3) at (72,0) {}; \draw (72,7) node {$v_3$};
\node[vertex,scale=0.5] (Lv4) at (83,0) {}; \draw (83,7) node {$v_4$};
\node[vertex,scale=0.5] (Lv5) at (94,0) {}; \draw (94,7) node {$v_5$};
\node[vertex,scale=0.5] (Lv6) at (105,0) {}; \draw (105,7) node {$v_6$};
\node[vertex,scale=0.5] (Lv7) at (116,0) {}; \draw (116,7) node {$v_7$};
\draw[rotate=0](105,0)ellipse(15 and 3.5); \draw (105,-8) node {$e_3$};
\draw[rotate=0](61,0)ellipse(15 and 3.5); \draw (61,-8) node {$e_1$};
\draw[rotate=0](83,0)ellipse(15 and 3.5); \draw (83,-8) node {$e_2$};

\end{tikzpicture}
\caption{The 3-uniform paths $\mathcal{T}$, $\mathcal{M}$, $\mathcal{L}$.}
\label{fig:3-paths}
\end{center}
\end{figure}

In 2024, Liu~\cite{Liu} obtained the following structural result on rainbow $\mathcal{L}$-free edge-colorings of $K^{(3)}_n$.

\begin{theorem}{\normalfont (\cite{Liu})}\label{th:Liu-stru}
For any integer $n\geq 10$, let $G$ be a rainbow $\mathcal{L}$-free edge-colored $K^{(3)}_n$ with $|C(G)|\geq 3$.
Then there exist three vertices $u, v, w$ such that $G- \{u,v,w\}$ is monochromatic.
\end{theorem}

In this paper, we obtain the following results for rainbow $\mathcal{T}$-free, rainbow $\mathcal{M}$-free and rainbow $\mathcal{L}$-free edge-colorings, which extend the above-mentioned results of Thomason and Wagner~\cite{ThWa} and of Liu~\cite{Liu}.

\begin{theorem}\label{th:tight}
For any integer $n\geq 5$, let $G$ be a rainbow $\mathcal{T}$-free edge-colored $K^{(3)}_n$ with $|C(G)|\geq 3$.
Then we can partition $V(G)$ into $|C(G)|-1$ parts $V_2, V_3, \ldots, V_{|C(G)|}$ such that, after renumbering the colors if necessary, $\{i\}\subseteq C(V_i)\subseteq \{1,i\}$ for every $i\in \{2, 3, \ldots, |C(G)|\}$, and every edge meeting at least two distinct parts has color 1.
\end{theorem}

\begin{theorem}\label{th:messy}
For any integer $n\geq 7$, let $G$ be a rainbow $\mathcal{M}$-free edge-colored $K^{(3)}_n$.
Then $|C(G)|\leq 2$.
\end{theorem}

\begin{theorem}\label{th:loose}
For any integer $n\geq 7$, let $G$ be a rainbow $\mathcal{L}$-free edge-colored $K^{(3)}_n$ with $|C(G)|\geq 3$.
Then at least one of the following statements holds:
\begin{itemize}
\item[{\rm (i)}] there exists a vertex $u\in V(G)$ such that $G-u$ is monochromatic;

\item[{\rm (ii)}] there exists an edge $e\in E(G)$ and a color $i\in C(G)$ with $c(e)\neq i$ such that every edge $f\in E(G)\setminus \{e\}$ with $c(f)\neq i$ satisfies $|f\cap e|= 2$.
\end{itemize}
\end{theorem}

\begin{remark}\label{re:theorems-1}
{\rm
(i) The condition $n\geq 7$ in Theorem~\ref{th:messy} is best possible.
To see this, note that $E(K^{(3)}_6)$ can be decomposed into $\frac{1}{2}{6\choose 3}=10$ perfect matchings.
We color the edges of $K_6^{(3)}$ using 10 distinct colors such that each color class is a perfect matching.
Since the first and third edges of every messy path $\mathcal{M}$ are disjoint and thus form a perfect matching in $E(K^{(3)}_6)$, such an edge-colored $K_6^{(3)}$ is rainbow $\mathcal{M}$-free.

(ii) Theorem~\ref{th:loose} implies that for $n\geq 7$, if $G$ is a rainbow $\mathcal{L}$-free edge-colored $K^{(3)}_n$ with $|C(G)|\geq 3$, then there exists a subset $U\subseteq V(G)$ with $|U|\leq 2$ such that $G- U$ is monochromatic.
This strengthens Theorem~\ref{th:Liu-stru}.
}
\end{remark}

\subsection{Applications to Ramsey theory}
\label{subsec:appl}

Given an $r$-graph $H$ and an integer $k\geq 2$, the {\it $k$-colored Ramsey number} $R_k(H)$ is defined as the minimum integer $n$ such that in every $k$-edge-coloring of the complete $r$-graph $K^{(r)}_{n}$, there is a monochromatic copy of $H$.
Given two $r$-graphs $H$ and $G$, the {\it constrained Ramsey number} (sometimes called the {\it rainbow Ramsey number}) $f(H,G)$ is the minimum integer $n$ such that in every edge-coloring of $K^{(r)}_n$ with any number of colors, there is either a monochromatic copy of $H$ or a rainbow copy of $G$.
We should point out that for general hypergraphs $H$ and $G$, the value $f(H,G)$ need not exist.
In fact, for ordinary graphs (i.e., $r=2$), $f(H,G)$ exists if and only if $H$ is a star or $G$ is a forest (see, for example, \cite{Eroh,JaJL}).
An analogous existence result for hypergraphs will be given in Section~\ref{sec:conclu} (see Proposition~\ref{prop:constrained}).

The constrained Ramsey number for ordinary graphs was introduced by Eroh~\cite{Eroh}, Jamison, Jiang and Ling~\cite{JaJL}, and Chen, Schelp and Wei~\cite{ChSW} independently in the early 2000s.
In the special case when $G=K_{1,t}$ is a star, the value $f(H,K_{1,t})$ was first studied by Gy\'{a}rf\'{a}s, Lehel, Schelp and Tuza~\cite{GLST} in the language of local Ramsey numbers.
For hypergraphs of uniformity at least 3, to the best of our knowledge, this problem was first studied by Liu~\cite{Liu} in 2024.
Related constrained Ramsey problems have been studied extensively; see \cite{AJMP,AxLe,BHHLM,GMSW,LiLiu,LoSu} and the references therein.

In the graph case, when $G$ is a path, the constrained Ramsey problem has received specific attention.
Let $P_{t}$ be the path on $t$ vertices.
For any tree $S$ on $s$ edges, Jamison, Jiang and Ling~\cite{JaJL} proved that $f(S,P_t)=\Omega(st)$, and they conjectured that $f(S,P_t)=O(st)$.
In 2009, Loh and Sudakov~\cite{LoSu} showed that $f(S, P_t) = O(st \log t)$.
Recently, Gishboliner, Milojevi\'{c}, Sudakov and Wigderson~\cite{GMSW} improved this to a nearly optimal upper bound which differs from the lower bound by a function of inverse-Ackermann type.
Jamison, Jiang and Ling~\cite{JaJL} also asked whether $f(S,T)$ is maximized by $f(P_{s+1}, P_{t+1})$ among all pairs of trees $S$ and $T$ with $s$ edges and $t$ edges, respectively.
In 2007, Gy\'{a}rf\'{a}s, Lehel and Schelp~\cite{GyLS} showed that for $t\in \{3,4\}$, the answer is negative.
Moreover, Gy\'{a}rf\'{a}s, Lehel and Schelp~\cite{GyLS} obtained the following result which exhibits a surprising connection between $f(H, P_t)$ and $R_{t-2}(H)$.

\begin{theorem}{\normalfont (\cite{GyLS})}\label{th:GyLS}
For any graph $H$ of order at least $5$, we have $f(H, P_4)=R_2(H)$.
If $H$ is a path, a cycle or a connected non-bipartite graph, then $f(H,P_5)=R_3(H)$.
\end{theorem}

Theorem~\ref{th:GyLS} was generalized by Li, Besse, Magnant, Wang and Watts~\cite{LBMWW}, who showed that for every graph $H$ that is connected or bipartite, we have $f(H,P_5)=R_3(H)$ (see also~\cite{LiLiu}).
For the 3-uniform loose path of length 3, Liu~\cite{Liu} obtained the following result.

\begin{theorem}{\normalfont (\cite{Liu})}\label{th:Liu-Ram}
For every 3-graph $H$ with $R_2(H)\geq \max\{|V(H)|+3, 10\}$, we have $f(H, \mathcal{L})=R_2(H)$.
\end{theorem}

As the first application of our structural results, we obtain the following results on hypergraph constrained Ramsey numbers.

\begin{theorem}\label{th:Ram-tight}
For every connected 3-graph $H$ with $R_2(H)\geq 5$, we have $f(H, \mathcal{T})=R_2(H)$.
\end{theorem}

\begin{theorem}\label{th:Ram-messy}
For every 3-graph $H$ with $R_2(H)\geq 7$, we have $f(H, \mathcal{M})=R_2(H)$.
\end{theorem}

\begin{theorem}\label{th:Ram-loose}
For every 3-graph $H$ with $R_2(H)\geq \max\{|V(H)|+1, 7\}$, we have $f(H, \mathcal{L})=R_2(H)$.
\end{theorem}

\begin{remark}\label{re:theorems-2}
{\rm
(i) The condition $R_2(H)\geq 5$ in Theorem~\ref{th:Ram-tight} is best possible.
Indeed, for any 3-graph $H$ with $|E(H)|\geq 2$, we have $f(H, \mathcal{T})\geq 5$ since a rainbow-colored $K_4^{(3)}$ contains neither a monochromatic $H$ nor a rainbow $\mathcal{T}$.
Hence, if $H$ satisfies $|E(H)|\geq 2$ and $R_2(H)\leq 4$, then $f(H, \mathcal{T})>R_2(H)$.
For example, let $H$ be the 3-graph with $V(H)=\{v_1, v_2, v_3, v_4\}$ and $E(H)=\{v_1v_2v_3, v_2v_3v_4\}$.
Then $H$ is connected and $R_2(H)= 4$.

(ii) The condition $R_2(H)\geq 7$ in Theorem~\ref{th:Ram-messy} is best possible.
To see this, let $H$ be a 3-graph with $|E(H)|\geq 2$ and $R_2(H)\leq 6$.
Since the 2-colored Ramsey number for the 3-uniform matching of size 2 is 7 (see~\cite{Lov}), $H$ contains no matching of size 2.
Then the edge-coloring of $K^{(3)}_6$ constructed in Remark~\ref{re:theorems-1}~(i) contains neither a rainbow $\mathcal{M}$ nor a monochromatic $H$.
This implies that $f(H, \mathcal{M})\geq 7>R_2(H)$.
Moreover, if an edge-colored $K^{(3)}_7$ contains no rainbow $\mathcal{M}$, then it is colored by at most two colors by Theorem~\ref{th:messy}.
Since $R_2(H)\leq 6<7$, there must be a monochromatic $H$ in every rainbow $\mathcal{M}$-free edge-colored $K^{(3)}_7$.
This implies that $f(H, \mathcal{M})= 7$ for any 3-graph $H$ with $|E(H)|\geq 2$ and $R_2(H)\leq 6$.
}
\end{remark}

Another rainbow generalization of the Ramsey number is the anti-Ramsey number, which was first introduced by Erd\H{o}s, Simonovits and S\'{o}s~\cite{ErSS} in the 1970s.
Given an $r$-uniform hypergraph $G$ and an integer $n\geq |V(G)|$, the {\it anti-Ramsey number} $ar(n, G)$ is the minimum integer $k$ such that in every edge-coloring of $K^{(r)}_n$ with at least $k$ colors, there is a rainbow copy of $G$.\footnote{We remark that in some of the literature, the anti-Ramsey number is also defined as the maximum integer $k'$ such that there exists a rainbow $G$-free edge-coloring of $K^{(r)}_n$ with exactly $k'$ colors. Under this alternative definition, the anti-Ramsey number is one less than the quantity defined here.}
Anti-Ramsey numbers for hypergraph paths, cycles and matchings have been studied in~\cite{GuLS,GuLP,Jin21DM,LTWY,OzYo,TaLY,XuSK}.
In particular, Tang, Li and Yan~\cite{TaLY} obtained the following result for $\mathcal{L}$.

\begin{theorem}{\normalfont (\cite{TaLY})}\label{th:TaLY}
For any integer $n\geq 32$, we have $ar(n,\mathcal{L})=n$.
\end{theorem}

As the second application of our structural results, we obtain the following results on hypergraph anti-Ramsey numbers.
In particular, for the loose path $\mathcal{L}$, our result improves Theorem~\ref{th:TaLY}.

\begin{theorem}\label{th:anti-Ram}
The following statements hold.
\begin{itemize}
\item[{\rm (i)}] For any integer $n\geq 5$, we have $ar(n,\mathcal{T})=\left\lfloor\frac{n}{3}\right\rfloor+2$.

\item[{\rm (ii)}] $ar(n,\mathcal{M})=
\left\{
   \begin{aligned}
    &11, & & \mbox{for $n=6$},\\
    &3, & & \mbox{for $n\geq 7$}.
   \end{aligned}
   \right.$

\item[{\rm (iii)}] For any integer $n\geq 7$, we have $ar(n,\mathcal{L})=n$.
\end{itemize}
\end{theorem}

\subsection{Multipartite extensions}
\label{subsec:multipartite}

In this subsection, we consider edge-colorings of complete 3-partite 3-graphs that contain no rainbow paths of length 3.
An $r$-graph $H$ is called {\it $r$-partite} if $V(H)$ can be partitioned into $r$ disjoint subsets $V_1, V_2, \ldots, V_r$ such that for every edge $e\in E(H)$, we have $|e\cap V_i|=1$ for each $i\in [r]$.
Note that a $2$-partite 2-graph is simply an ordinary bipartite graph.
In \cite{LiWL}, Li, Wang and Liu characterized the structures of edge-colored complete bipartite graphs that contain no rainbow paths of length 3 and length 4, respectively.
As applications, Li and Liu~\cite{LiLiu} obtained several constrained Ramsey-type results in the setting of bipartite graphs.
Let $V_1, V_2, \ldots, V_r$ be $r$ disjoint vertex sets with $|V_i|=n_i$ for each $i\in [r]$.
The {\it complete $r$-partite $r$-graph} $K_{n_1,\ldots,n_r}^{(r)}$ with partite sets $V_1, V_2, \ldots, V_r$ is defined as the $r$-partite $r$-graph whose edge set consists of all the $r$-element subsets $e$ of $V_1\cup V_2\cup \cdots \cup V_r$ with $|e\cap V_i|=1$ for each $i\in [r]$.
As a combined extension of our Theorems~\ref{th:tight}, \ref{th:messy}, \ref{th:loose} and the results of Li, Wang and Liu~\cite{LiWL}, we obtain the following results.

\begin{theorem}\label{th:multi-tight}
For any integer $n\geq 3$, let $G$ be a rainbow $\mathcal{T}$-free edge-colored $K^{(3)}_{n,n,n}$ with $|C(G)|\geq 3$.
Let $V_1, V_2, V_3$ be the partite sets of $G$.
Then at least one of the following statements holds:
\begin{itemize}
\item[{\rm (i)}] there exists an $\ell\in [3]$ such that $V_{\ell}$ can be partitioned into $|C(G)|$ parts $V_{\ell,1}, V_{\ell,2}, \ldots, V_{\ell,|C(G)|}$ so that for each $i\in [|C(G)|]$, all edges containing a vertex in $V_{\ell,i}$ are of color $i$.

\item[{\rm (ii)}] for every $\ell\in [3]$, we can partition $V_{\ell}$ into $|C(G)|-1$ parts $V_{\ell,2}, V_{\ell,3}, \ldots, V_{\ell,|C(G)|}$ such that, after renumbering the colors if necessary, all edges of $G[V_{1,i}\cup V_{2,i}\cup V_{3,i}]$ are of color 1 or $i$ for every $i\in \{2, 3, \ldots, |C(G)|\}$, and all the remaining edges are of color 1.
\end{itemize}
\end{theorem}

\begin{theorem}\label{th:multi-messy}
For any integer $n\geq 3$, let $G$ be a rainbow $\mathcal{M}$-free edge-colored $K^{(3)}_{n,n,n}$ with $|C(G)|\geq 3$.
Let $V_1, V_2, V_3$ be the partite sets of $G$.
Then there exists an $\ell\in [3]$ such that $V_{\ell}$ can be partitioned into $|C(G)|$ parts $V_{\ell,1}, V_{\ell,2}, \ldots, V_{\ell,|C(G)|}$ so that for each $i\in [|C(G)|]$, all edges containing a vertex in $V_{\ell,i}$ are of color $i$.
\end{theorem}

\begin{theorem}\label{th:multi-loose}
For any integer $n\geq 3$, let $G$ be a rainbow $\mathcal{L}$-free edge-colored $K^{(3)}_{n,n,n}$ with $|C(G)|\geq 3$.
Let $V_1, V_2, V_3$ be the partite sets of $G$.
Then at least one of the following statements holds $($after renumbering the colors or permuting the partite sets if necessary$)$:
\begin{itemize}
\item[{\rm (i)}] there exist two vertices $x_1\in V_1$, $y_1\in V_2$ and a color $i\in C(G)$ such that every edge $e\in E(G)$ with $c(e)\neq i$ satisfies $\{x_1, y_1\}\subseteq e$;

\item[{\rm (ii)}] $|C(G)|=3$, and there exists a unique edge $e$ of color 1 such that every edge $f$ of color 2 satisfies $|e\cap f|=2$, and all the remaining edges are of color 3;

\item[{\rm (iii)}] $|C(G)|=3$, and there exist five vertices $x_1\in V_1$, $y_1, y_2\in V_2$, $z_1, z_2\in V_3$ such that $c(x_1y_1z_1)=c(x_1y_2z_2)=1$, $c(x_1y_1z_2)=c(x_1y_2z_1)=2$, and all the remaining edges are of color 3.
\end{itemize}
\end{theorem}

As an application, we next consider constrained Ramsey-type problems for rainbow $\mathcal{T}$, $\mathcal{M}$ and $\mathcal{L}$ with respect to 3-partite 3-graphs.
Given an $r$-partite $r$-graph $H$ and an integer $k\geq 2$, the {\it $k$-colored $r$-partite Ramsey number} $R'_k(H)$ is defined as the minimum integer $n$ such that in every $k$-edge-coloring of the complete $r$-partite $r$-graph $K^{(r)}_{n, \ldots, n}$, there is a monochromatic copy of $H$.
Given two $r$-partite $r$-graphs $H$ and $G$, the {\it $r$-partite constrained Ramsey number} $f'(H,G)$ is the minimum integer $n$ such that in every edge-coloring of $K^{(r)}_{n, \ldots, n}$ with any number of colors, there is either a monochromatic copy of $H$ or a rainbow copy of $G$.

We now determine for which $r$-partite $r$-graphs $H$ and $G$ the $r$-partite constrained Ramsey number $f'(H,G)$ exists.
This existence problem is closely related to a multipartite version of the Erd\H{o}s-Rado Canonical Ramsey Theorem~\cite{ErRa,Rado954},
which asserts that every edge-coloring of a sufficiently large $K^{(r)}_{n, \ldots, n}$ contains a large subgraph $K^{(r)}_{t, \ldots, t}$ on which the coloring is one of several canonical types.
We now introduce the definition of canonical colorings for $r$-partite $r$-graphs.

\begin{definition}{\normalfont (see~\cite[Definition~1.1]{CaSS})}\label{def:canonical-multipar}
{\rm
For an $r$-partite $r$-graph $H$ with partite sets $V_1, V_2, \ldots, V_r$, a set $J\subseteq [r]$ and an edge $e\in E(H)$, let $e_J\colonequals e\cap (\bigcup_{j\in J}V_j)$.
An edge-coloring of $H$ is called {\it $J$-canonical} if for all edges $e,e'\in E(H)$, we have $c(e)=c(e')$ if and only if $e_J=e'_J$.
}
\end{definition}

Note that an $\emptyset$-canonical coloring is a monochromatic coloring, and a $[r]$-canonical coloring is a rainbow coloring.
Moreover, if $H$ is the complete $r$-partite $r$-graph $K^{(r)}_{n, \ldots, n}$ and $|J|=|J'|$, then a $J$-canonical coloring of $H$ is isomorphic to a $J'$-canonical coloring of $H$ (after renumbering the colors if necessary).
We will use the following multipartite version of the Erd\H{o}s-Rado Canonical Ramsey Theorem.

\begin{theorem}[Multipartite version of the Erd\H{o}s-Rado Canonical Ramsey Theorem~\cite{Rado954}]\label{th:canonical-multipar}
For any integers $r\geq 2$ and $t\geq 1$, there exists an integer $n$ such that in every edge-coloring of $K^{(r)}_{n, \ldots, n}$, there is a $J$-canonical copy of $K^{(r)}_{t, \ldots, t}$ for some $J\subseteq [r]$.
\end{theorem}

By Theorem~\ref{th:canonical-multipar}, we can deduce the following existence result for $f'(H,G)$.
For completeness, we will provide a proof in Section~\ref{sec:multipartite}.

\begin{proposition}\label{prop:constrained-multipar}
Let $H$ and $G$ be $r$-partite $r$-graphs.
Then the $r$-partite constrained Ramsey number $f'(H,G)$ exists if and only if there exists a positive integer $t$ such that for every subset $J\subseteq [r]$,
$H$ occurs as a monochromatic subgraph of a $J$-canonical edge-coloring of $K^{(r)}_{t, \ldots, t}$
or $G$ occurs as a rainbow subgraph of a $J$-canonical edge-coloring of $K^{(r)}_{t, \ldots, t}$.
\end{proposition}

Now we consider $3$-partite constrained Ramsey numbers for rainbow $\mathcal{T}$, $\mathcal{M}$ and $\mathcal{L}$, respectively.
Note that for all $3$-partite 3-graphs $H$ and $G$, $H$ occurs as a monochromatic subgraph of the $\emptyset$-canonical edge-coloring of $K^{(3)}_{|V(H)|, |V(H)|, |V(H)|}$,
and $G$ occurs as a rainbow subgraph of the $[3]$-canonical edge-coloring of $K^{(3)}_{|V(G)|, |V(G)|, |V(G)|}$.
Moreover, one can easily check that for any $J\subseteq [3]$ with $|J|=1$, $\mathcal{T}$ and $\mathcal{M}$ are not isomorphic to a rainbow subgraph in the $J$-canonical edge-coloring of $K^{(3)}_{t, t, t}$ for any $t$, but $\mathcal{L}$ occurs as a rainbow subgraph of the $J$-canonical edge-coloring of $K^{(3)}_{3, 3, 3}$.
Furthermore, for any $J\subseteq [3]$ with $|J|=2$, each of $\mathcal{T}$, $\mathcal{M}$ and $\mathcal{L}$ occurs as a rainbow subgraph of the $J$-canonical edge-coloring of $K^{(3)}_{3, 3, 3}$.
Therefore, by Proposition~\ref{prop:constrained-multipar}, $f'(H,\mathcal{L})$ exists for all $3$-partite 3-graphs $H$, while $f'(H,\mathcal{T})$ and $f'(H,\mathcal{M})$ exist if and only if $H$ occurs as a monochromatic subgraph of a $J$-canonical edge-coloring of some $K^{(3)}_{t, t, t}$ with $|J|=1$.
We have the following results as applications of Theorems~\ref{th:multi-tight}, \ref{th:multi-messy} and \ref{th:multi-loose}.

\begin{theorem}\label{th:Ram-tight-messy-multipar}
If $H$ occurs as a monochromatic subgraph of a $J$-canonical edge-coloring of $K^{(3)}_{t, t, t}$ with $|J|=1$ for some positive integer $t$, then $f'(H, \mathcal{T})=f'(H, \mathcal{M})=R'_2(H)$.
\end{theorem}

Given a 3-partite 3-graph $H$, let $t(H)$ be the minimum integer $t$ such that $H\subseteq K^{(3)}_{t,t,t}$.

\begin{theorem}\label{th:Ram-loose-multipar}
For every 3-partite 3-graph $H$ with $R'_2(H)\geq t(H)+1$, we have $f'(H, \mathcal{L})=R'_2(H)$.
\end{theorem}

We next consider a multipartite version of the anti-Ramsey number.
Given an $r$-partite $r$-graph $G$, the {\it anti-Ramsey number} $ar(K^{(r)}_{n, \ldots, n}, G)$ is the minimum integer $k$ such that in every edge-coloring of $K^{(r)}_{n, \ldots, n}$ with at least $k$ colors, there is a rainbow copy of $G$.
To the best of our knowledge, anti-Ramsey numbers $ar(K^{(r)}_{n, \ldots, n}, G)$ have previously been determined only for matchings (see~\cite{Jin21DM,XuSK}).
Combining Theorems~\ref{th:multi-tight}, \ref{th:multi-messy}, \ref{th:multi-loose} with suitable lower-bound constructions, we obtain the following exact values.

\begin{theorem}\label{th:anti-Ram-multipar}
For any integer $n\geq 3$, we have $ar(K^{(3)}_{n, n, n},\mathcal{T})=ar(K^{(3)}_{n, n, n},\mathcal{L})=n+2$ and $ar(K^{(3)}_{n, n, n},\mathcal{M})=n+1$.
\end{theorem}

\vspace{0.4cm}
\noindent{\bf Additional notation.}~
Given two hypergraphs $G$ and $H$, let $G\cup H$ be the disjoint union of $G$ and $H$.
For any integer $\ell\geq 3$, the $3$-uniform {\it loose cycle} $C^{(3)}_{\ell}$ of length $\ell$ is a $3$-graph with vertex set $\{v_1, v_2, \ldots, v_{2\ell}\}$ and edge set $\{e_1, e_2, \ldots, e_{\ell}\}$ such that $e_i=\{v_{2(i-1)+1}, v_{2(i-1)+2}, v_{2(i-1)+3}\}$ for each $i\in [\ell]$, where $v_{2\ell+1}\colonequals v_1$.
For any integer $\ell\geq 1$, the $3$-uniform {\it loose star} $S^{(3)}_{\ell}$ of size $\ell$ is a $3$-graph with vertex set $\{v_0, v_1, v_2, \ldots, v_{2\ell}\}$ and edge set $\{e_1, e_2, \ldots, e_{\ell}\}$ such that $e_i=\{v_0, v_{2i-1}, v_{2i}\}$ for each $i\in [\ell]$.
Let $\mathbb{S}^{(3)}_{\ell}$ be the 3-graph with vertex set $\{u, v, v_1, v_2, \ldots, v_{\ell}\}$ and edge set $\{e_1, e_2, \ldots, e_{\ell}\}$ such that $e_i=\{u, v, v_i\}$ for each $i\in [\ell]$.
Note that both $S^{(3)}_1$ and $\mathbb{S}^{(3)}_{1}$ consist of a single edge, $S^{(3)}_2=P^{(3)}_2$ and $\mathbb{S}^{(3)}_{2}=\mathbb{P}^{(3)}_2$.
See Figure~\ref{fig:3-stars} for an illustration of these 3-graphs.
Given an edge-colored (hyper)graph $G$ and a vertex $v\in V(G)$, let $d^{c}(v)$ be the {\it color degree} of $v$ in $G$, that is, the number of distinct colors on edges incident with $v$.
Let $\Delta^{c}(G)\colonequals \max_{v\in V(G)}d^{c}(v)$ be the {\it maximum color degree} of $G$.
Finally, we remark that for an edge $e=\{v_i, v_j, v_k\}$, we will also use $v_iv_jv_k$, $v_iv_kv_j$, $v_jv_iv_k$, $v_jv_kv_i$, $v_kv_iv_j$ or $v_kv_jv_i$ to denote this edge.
For the same edge, we may choose any of these representations depending on the context.

\begin{figure}[htbp]
\begin{center}
\begin{tikzpicture}[scale=0.06,auto,swap]
\tikzstyle{vertex}=[circle,draw=black,fill=black]
\draw (0,-36) node {$S^{(3)}_2$};
\node[vertex,scale=0.5] (Sv1) at (0,18) {}; \draw (-8,18) node {$v_2$};
\node[vertex,scale=0.5] (Sv2) at (0,9) {}; \draw (-8,9) node {$v_1$};
\node[vertex,scale=0.5] (Sv3) at (0,0) {}; \draw (-8,0) node {$v_0$};
\node[vertex,scale=0.5] (Sv4) at (0,-9) {}; \draw (-8,-9) node {$v_3$};
\node[vertex,scale=0.5] (Sv5) at (0,-18) {}; \draw (-8,-18) node {$v_4$};
\draw[rotate=0](0,9)ellipse(3.5 and 14); \draw (8,10) node {$e_1$};
\draw[rotate=0](0,-9)ellipse(3.5 and 14); \draw (8,-10) node {$e_2$};

\draw (59,-36) node {$S^{(3)}_3$};
\node[vertex,scale=0.5] (SSv0) at (59,-18) {}; \draw (59,-26) node {$v_0$};
\node[vertex,scale=0.5] (SSv2) at (29,18) {}; \draw (22,17) node {$v_2$};
\node[vertex,scale=0.5] (SSv1) at (44,0) {}; \draw (35,-0.5) node {$v_1$};
\node[vertex,scale=0.5] (SSv4) at (59,18) {}; \draw (52,17) node {$v_4$};
\node[vertex,scale=0.5] (SSv3) at (59,0) {}; \draw (52.5,-0.5) node {$v_3$};
\node[vertex,scale=0.5] (SSv6) at (89,18) {}; \draw (96,17) node {$v_6$};
\node[vertex,scale=0.5] (SSv5) at (74,0) {}; \draw (83,-0.5) node {$v_5$};
\coordinate (X3) at (44,0);
\draw[rotate around={130:(X3)}](X3)ellipse(28.5 and 3.5); \draw (26,25.5) node {$e_1$};
\coordinate (X4) at (59,0);
\draw[rotate around={90:(X4)}] (X4) ellipse (22.5 and 3.5); \draw (59,25.5) node {$e_2$};
\coordinate (X5) at (74,0);
\draw[rotate around={50:(X5)}] (X5) ellipse (28.5 and 3.5); \draw (92,25.5) node {$e_3$};

\draw (119,-36) node {$\mathbb{S}^{(3)}_{2}$};
\node[vertex,scale=0.5] (SSSv1) at (119,18) {}; \draw (111,18) node {$v_1$};
\node[vertex,scale=0.5] (SSSv2) at (119,6) {}; \draw (111,6) node {$u$};
\node[vertex,scale=0.5] (SSSv3) at (119,-6) {}; \draw (111,-6) node {$v$};
\node[vertex,scale=0.5] (SSSv4) at (119,-18) {}; \draw (111,-18) node {$v_2$};
\draw[rotate=0](119,6)ellipse(3.5 and 17); \draw (127,10) node {$e_1$};
\draw[rotate=0](119,-6)ellipse(3.5 and 17); \draw (127,-10) node {$e_2$};

\draw (157,-36) node {$\mathbb{S}^{(3)}_{3}$};
\node[vertex,scale=0.5] (SSSSv1) at (165,18) {}; \draw (173,18) node {$v_1$};
\node[vertex,scale=0.5] (SSSSv2) at (149,6) {}; \draw (141,6) node {$u$};
\node[vertex,scale=0.5] (SSSSv3) at (149,-6) {}; \draw (141,-6) node {$v$};
\node[vertex,scale=0.5] (SSSSv4) at (165,0) {}; \draw (173,0) node {$v_2$};
\node[vertex,scale=0.5] (SSSSv4) at (165,-18) {}; \draw (173,-18) node {$v_3$};
\draw[rounded corners=5pt] (144,7) -- (144,-6) -- (150,-10) -- (165,13) -- (169,18) -- (165,22) -- (149,11) -- cycle; \draw (165,25) node {$e_1$};
\draw[rounded corners=5pt] (145,5) -- (145,-7) -- (150,-11) -- (165,-22) -- (169,-18) -- (165,-13) -- (149,10.5) -- cycle; \draw (165,5.5) node {$e_2$};
\draw[rounded corners=4pt] (146,6) -- (146,-6) -- (149,-9) -- (165,-3) -- (168.5,0) -- (165,3) -- (149,9) -- cycle; \draw (165,-25) node {$e_3$};

\draw (216,-36) node {$C^{(3)}_3$};
\node[vertex,scale=0.5] (cv1) at (216,18) {}; \draw (216,26) node {$v_1$};
\node[vertex,scale=0.5] (cv4) at (216,-18) {}; \draw (216,-26) node {$v_4$};
\node[vertex,scale=0.5] (cv3) at (195,-18) {}; \draw (190,-24) node {$v_3$};
\node[vertex,scale=0.5] (cv5) at (237,-18) {}; \draw (242,-24) node {$v_5$};
\node[vertex,scale=0.5] (cv2) at (205.5,0) {}; \draw (196,0) node {$v_2$};
\node[vertex,scale=0.5] (cv6) at (226.5,0) {}; \draw (236,0) node {$v_6$};
\draw[rotate=0](216,-18)ellipse(26 and 3.5);\draw (216,-11) node {$e_2$};
\coordinate (X1) at (205,-1);
\draw[rotate around={60:(X1)}] (X1) ellipse (26 and 3.5);\draw (202,10) node {$e_1$};
\coordinate (X2) at (227,-1);
\draw[rotate around={120:(X2)}] (X2) ellipse (26 and 3.5); \draw (230,10) node {$e_3$};
\end{tikzpicture}
\caption{The 3-graphs $S^{(3)}_2$, $S^{(3)}_3$, $\mathbb{S}^{(3)}_{2}$, $\mathbb{S}^{(3)}_{3}$ and $C^{(3)}_3$.}
\label{fig:3-stars}
\end{center}
\end{figure}

\vspace{0.2cm}
The remainder of this paper is organized as follows.
In Sections~\ref{sec:tight}, \ref{sec:messy} and \ref{sec:loose}, we prove Theorems~\ref{th:tight}, \ref{th:messy} and \ref{th:loose}, respectively.
In Section~\ref{sec:Ramsey}, we present our proofs of Theorems~\ref{th:Ram-tight}, \ref{th:Ram-messy}, \ref{th:Ram-loose} and \ref{th:anti-Ram}.
In Section~\ref{sec:multipartite}, we prove Theorems~\ref{th:multi-tight}, \ref{th:multi-messy}, \ref{th:multi-loose}, \ref{th:Ram-tight-messy-multipar}, \ref{th:Ram-loose-multipar}, \ref{th:anti-Ram-multipar} and Proposition~\ref{prop:constrained-multipar}.
Finally, in Section~\ref{sec:conclu}, we conclude this paper by presenting several open problems and remarks.

\section{Tight path--Proof of Theorem~\ref{th:tight}}
\label{sec:tight}

In this section, we prove the structural result (Theorem~\ref{th:tight}) for the tight path $\mathcal{T}$.
We first state and prove a simple observation and two technical lemmas.

\begin{observation}\label{obs:S2}
For any integer $n\geq 5$, if $G$ is an edge-colored $K_{n}^{(3)}$ with $|C(G)|\geq 2$, then $G$ contains both a rainbow $S^{(3)}_2$ and a rainbow $\mathbb{S}^{(3)}_{2}$.
\end{observation}

\begin{proof} Let $G$ be an edge-colored $K_{n}^{(3)}$ with $|C(G)|\geq 2$ and $V(G)=\{v_1, v_2, \ldots, v_n\}$, where $n\geq 5$.
Since $|C(G)|\geq 2$, there exist two edges $e$ and $f$ of distinct colors such that $e\cap f\neq \emptyset$.
Thus $G$ contains either a rainbow $S^{(3)}_2$ or a rainbow $\mathbb{S}^{(3)}_{2}$.
If $G$ contains a rainbow $S^{(3)}_2$, say $\{v_1v_2v_3, v_1v_4v_5\}$, then at least one of $\{v_1v_2v_3, v_1v_2v_4\}$ and $\{v_1v_4v_5, v_1v_4v_2\}$ is a rainbow $\mathbb{S}^{(3)}_{2}$,
so $G$ contains both a rainbow $S^{(3)}_2$ and a rainbow $\mathbb{S}^{(3)}_{2}$ in this case.
If $G$ contains a rainbow $\mathbb{S}^{(3)}_{2}$, say $\{v_1v_2v_3, v_1v_2v_4\}$, then at least one of $\{v_3v_1v_2, v_3v_4v_5\}$ and $\{v_4v_1v_2, v_4v_3v_5\}$ is a rainbow $S^{(3)}_2$,
so $G$ contains both a rainbow $S^{(3)}_2$ and a rainbow $\mathbb{S}^{(3)}_{2}$.
\end{proof}

\begin{lemma}\label{le:tight-1}
Let $G$ be a rainbow $\mathcal{T}$-free edge-coloring of $K_{n}^{(3)}$.
Then $G$ contains no rainbow $C^{(3)}_3$, $S^{(3)}_3$ or $\mathbb{S}^{(3)}_{3}$.
\end{lemma}

\begin{proof} Let $V(G)=\{v_1, v_2, \ldots, v_n\}$.
We first show that $G$ contains no rainbow $C^{(3)}_3$.
Suppose for a contradiction that $G$ contains a rainbow $C^{(3)}_3$, say $\{v_1v_2v_3, v_3v_4v_5, v_5v_6v_1\}$.
Then no matter what color is assigned to the edge $v_1v_3v_5$, at least one of $\{v_4v_3v_5, v_3v_5v_1, v_5v_1v_6\}$, $\{v_2v_3v_1, v_3v_1v_5, v_1v_5v_6\}$ and $\{v_2v_1v_3, v_1v_3v_5, v_3v_5v_4\}$ is a rainbow $\mathcal{T}$, a contradiction.
Hence, $G$ contains no rainbow $C^{(3)}_3$.

We next show that $G$ contains no rainbow $S^{(3)}_3$.
Suppose for a contradiction that $G$ contains a rainbow $S^{(3)}_3$, say $\{v_1v_2v_3, v_1v_4v_5, v_1v_6v_7\}$.
Then no matter what color is assigned to the edge $v_2v_4v_6$, at least one of $\{v_4v_2v_6, v_6v_7v_1, v_1v_5v_4\}$, $\{v_2v_4v_6, v_6v_7v_1, v_1v_3v_2\}$ and $\{v_2v_6v_4, v_4v_5v_1, v_1v_3v_2\}$ is a rainbow $C^{(3)}_3$.
But $G$ contains no rainbow $C^{(3)}_3$ by the preceding argument.
Hence, $G$ contains no rainbow $S^{(3)}_3$.

Finally, we show that $G$ contains no rainbow $\mathbb{S}^{(3)}_{3}$.
Suppose for a contradiction that $G$ contains a rainbow $\mathbb{S}^{(3)}_{3}$, say $\{v_1v_2v_3, v_1v_2v_4, v_1v_2v_5\}$, where $c(v_1v_2v_3)=1$, $c(v_1v_2v_4)=2$ and $c(v_1v_2v_5)=3$.
Then $c(v_1v_3v_4)=3$, since otherwise at least one of $\{v_5v_2v_1, v_2v_1v_4, v_1v_4v_3\}$ and $\{v_5v_2v_1, v_2v_1v_3, v_1v_3v_4\}$ is a rainbow $\mathcal{T}$.
Now if $c(v_2v_5v_4)\neq 1$, then at least one of $\{v_3v_1v_2, v_1v_2v_5, v_2v_5v_4\}$ and $\{v_3v_1v_2, v_1v_2v_4, v_2v_4v_5\}$ is a rainbow $\mathcal{T}$;
if $c(v_2v_5v_4)=1$, then $\{v_3v_1v_4, v_1v_4v_2, v_4v_2v_5\}$ is a rainbow $\mathcal{T}$.
In either case, this is a contradiction.
Hence, $G$ contains no rainbow $\mathbb{S}^{(3)}_{3}$.
\end{proof}

\begin{lemma}\label{le:tight-2}
Let $G$ be a rainbow $\mathcal{T}$-free edge-coloring of $K_{n}^{(3)}$.
Let $\{e_1, e_2\}$ be a rainbow $S^{(3)}_2$ in $G$.
Then for every edge $f\in E(G)$ with $f\cap e_1\neq \emptyset$ and $f\cap e_2\neq \emptyset$, we have $c(f)\in \{c(e_1), c(e_2)\}$.
\end{lemma}

\begin{proof}
Let $V(G)=\{v_1, v_2, \ldots, v_n\}$.
Without loss of generality, we may assume that $e_1=v_1v_2v_3$, $e_2=v_1v_4v_5$, $c(e_1)=1$ and $c(e_2)=2$.
Suppose for a contradiction that there exists an edge $f\in E(G)$ with $f\cap e_1\neq \emptyset$ and $f\cap e_2\neq \emptyset$ but $c(f)\notin \{1,2\}$, say $c(f)=3$.
Note that $c(v_1v_iv_j)\in \{1,2\}$ for all $i\in \{2,3\}$ and $j\in \{4,5\}$, since otherwise $e_1, e_2$ and an appropriate such cross-edge form a rainbow $\mathcal{T}$.

Let $X=f\cap \{v_1, v_2, \ldots, v_5\}$.
We first show that $|X|\leq 2$.
Suppose not, that is, $f\subseteq \{v_1, v_2, \ldots, v_5\}$.
By the preceding argument, $v_1\notin f$ in this case.
Without loss of generality, we may assume that $f=v_2v_3v_4$.
Then $c(v_2v_4v_5)=3$, since otherwise at least one of $\{v_3v_2v_4, v_2v_4v_5, v_4v_5v_1\}$ and $\{v_5v_4v_2, v_4v_2v_3, v_2v_3v_1\}$ is a rainbow $\mathcal{T}$, a contradiction.
Recall that $c(v_1v_2v_4)\in \{1,2\}$.
Then at least one of $\{v_3v_2v_4, v_2v_4v_1, v_4v_1v_5\}$ and $\{v_5v_4v_2, v_4v_2v_1, v_2v_1v_3\}$ is a rainbow $\mathcal{T}$.
This contradiction implies that $|X|\leq 2$.

If $|X|=1$, then since $f\cap e_1\neq \emptyset$ and $f\cap e_2\neq \emptyset$, we have $X=\{v_1\}$.
But then $\{e_1, e_2, f\}$ is a rainbow $S^{(3)}_3$, contradicting Lemma~\ref{le:tight-1}.
If $|X|=2$ and $v_1\notin X$, then $\{e_1, e_2, f\}$ is a rainbow $C^{(3)}_3$, also contradicting Lemma~\ref{le:tight-1}.
Hence, we have $|X|=2$ and $v_1\in X$, say $f=v_1v_2v_6$ by symmetry.
By Lemma~\ref{le:tight-1}, $G$ contains no rainbow $\mathbb{S}^{(3)}_{3}$, so $c(v_1v_2v_4)\in \{1,3\}$.
Combining with $c(v_1v_2v_4)\in \{1,2\}$, we have $c(v_1v_2v_4)=1$.
But then $\{v_6v_2v_1, v_2v_1v_4, v_1v_4v_5\}$ is a rainbow $\mathcal{T}$, a contradiction.
This completes the proof of Lemma~\ref{le:tight-2}.
\end{proof}

Now we have all ingredients to present our proof of Theorem~\ref{th:tight}.

\vspace{0.3cm}
\noindent{\bf Proof of Theorem~\ref{th:tight}.}~
Let $G$ be a rainbow $\mathcal{T}$-free edge-colored $K^{(3)}_n$ with $|C(G)|\geq 3$.
We first prove a claim related to the maximum color degree $\Delta^{c}(G)$ of $G$.

\begin{claim}\label{cl:tight-1}
$\Delta^{c}(G)\leq 2.$
\end{claim}

\begin{proof}
Suppose for a contradiction that there exists a vertex $v_1$ with $d^c(v_1)\geq 3$.
Without loss of generality, we may assume that $e_1, e_2, e_3$ are three edges with $v_1\in e_1\cap e_2\cap e_3$ and $c(e_i)=i$ for each $i\in [3]$.
By Lemma~\ref{le:tight-1}, $G$ contains no rainbow $\mathbb{S}^{(3)}_{3}$, so $|e_1\cap e_2\cap e_3|\leq 1$.
Hence, $e_1\cap e_2\cap e_3=\{v_1\}$.
If $(e_1\cap e_2)\setminus \{v_1\}= \emptyset$, then $\{e_1, e_2\}$ is a rainbow $S^{(3)}_2$.
Since $e_3\cap e_1\neq \emptyset$, $e_3\cap e_2\neq \emptyset$ and $c(e_3)\notin \{c(e_1), c(e_2)\}$, we can deduce a contradiction by Lemma~\ref{le:tight-2}.
Thus $(e_1\cap e_2)\setminus \{v_1\}\neq \emptyset$, and by symmetry, we have $(e_1\cap e_3)\setminus \{v_1\}\neq \emptyset$ and $(e_2\cap e_3)\setminus \{v_1\}\neq \emptyset$.
Combining with these properties, we may assume that $e_1=v_1v_2v_3$, $e_2=v_1v_2v_4$ and $e_3=v_1v_3v_4$.
Let $v_5\in V(G)\setminus \{v_1, v_2, v_3, v_4\}$.
Since $G$ contains no rainbow $\mathbb{S}^{(3)}_{3}$, we have $c(v_1v_2v_5)\in \{1,2\}$, say $c(v_1v_2v_5)=1$ by symmetry.
Then $\{v_5v_2v_1, v_2v_1v_4, v_1v_4v_3\}$ is a rainbow $\mathcal{T}$, a contradiction.
\end{proof}

By Observation~\ref{obs:S2}, there exists a vertex $v_1$ of color degree at least 2 in $G$.
Combining with Claim~\ref{cl:tight-1}, we have $d^c(v_1)=2$.
Without loss of generality, we may assume that $v_1$ is incident with edges of colors 1 and 2.
Since $|C(G)|\geq 3$, there exists a vertex $v_2\in V(G)\setminus \{v_1\}$ that is incident with an edge of some color in $C(G)\setminus \{1,2\}$.
Since $d^c(v_2)\leq \Delta^{c}(G)\leq 2$ and $v_1$ is only incident with edges of colors 1 and 2, the edges containing both $v_1$ and $v_2$ are either all of color 1 or all of color 2, say color 1.
That is, for every vertex $v\in V(G)\setminus \{v_1, v_2\}$, we have $c(v_1v_2v)=1$.
Hence, every vertex in $G$ is incident with an edge of color 1.
Combining with Claim~\ref{cl:tight-1}, we can partition $V(G)$ into $|C(G)|$ parts $U_1, U_2, \ldots, U_{|C(G)|}$ (where $U_1$ may be empty) such that
\begin{itemize}
\item for each vertex $v\in U_1$, $v$ is only incident with edges of color 1; and

\item for each $i\in \{2, 3, \ldots, |C(G)|\}$ and each vertex $v\in U_i$, $v$ is only incident with edges of colors 1 and $i$.
\end{itemize}
Let $V_2=U_1\cup U_2$, and $V_i=U_i$ for each $i\in \{3, \ldots, |C(G)|\}$.
Then $\{i\}\subseteq C(V_i)\subseteq \{1,i\}$ for every $i\in \{2, 3, \ldots, |C(G)|\}$, and every edge meeting at least two distinct parts has color 1.
This completes the proof of Theorem~\ref{th:tight}.
\hfill$\square$

\section{Messy path--Proof of Theorem~\ref{th:messy}}
\label{sec:messy}

In this section, we prove the structural result (Theorem~\ref{th:messy}) for the messy path $\mathcal{M}$.

\vspace{0.3cm}
\noindent{\bf Proof of Theorem~\ref{th:messy}.}~
Let $G$ be a rainbow $\mathcal{M}$-free edge-colored $K^{(3)}_n$, where $V(G)=\{v_1, v_2, \ldots, v_n\}$ and $n\geq 7$.
For a contradiction, suppose that $|C(G)|\geq 3$.

\begin{claim}\label{cl:messy-1}
$G$ contains neither a rainbow $\mathbb{S}^{(3)}_{2}\cup \mathbb{S}^{(3)}_{1}$ nor a rainbow $\mathbb{S}^{(3)}_{3}$.
\end{claim}

\begin{proof}
Suppose for a contradiction that $\{v_1v_2v_3, v_1v_2v_4, v_5v_6v_7\}$ is a rainbow $\mathbb{S}^{(3)}_{2}\cup \mathbb{S}^{(3)}_{1}$, where $c(v_1v_2v_3)=1$, $c(v_1v_2v_4)=2$ and $c(v_5v_6v_7)=3$.
If $c(v_3v_6v_7)\notin \{1,2\}$, then $\{v_4v_1v_2, v_1v_2v_3, v_3v_6v_7\}$ is a rainbow $\mathcal{M}$;
if $c(v_3v_6v_7)\notin \{1,3\}$, then $\{v_5v_6v_7, v_6v_7v_3, v_3v_2v_1\}$ is a rainbow $\mathcal{M}$.
Thus $c(v_6v_7v_3)=1$.
Now if $c(v_6v_7v_1)=3$, then $\{v_3v_6v_7, v_6v_7v_1, v_1v_2v_4\}$ is a rainbow $\mathcal{M}$;
if $c(v_6v_7v_1)\neq 3$, then at least one of $\{v_5v_6v_7, v_6v_7v_1, v_1v_2v_3\}$ and $\{v_5v_6v_7, v_6v_7v_1, v_1v_2v_4\}$ is a rainbow $\mathcal{M}$.
In either case, this is a contradiction.
Thus $G$ contains no rainbow $\mathbb{S}^{(3)}_{2}\cup \mathbb{S}^{(3)}_{1}$.

Suppose now $\{v_1v_2v_3, v_1v_2v_4, v_1v_2v_5\}$ is a rainbow $\mathbb{S}^{(3)}_{3}$, where $c(v_1v_2v_3)=1$, $c(v_1v_2v_4)=2$ and $c(v_1v_2v_5)=3$.
Since $G$ contains no rainbow $\mathbb{S}^{(3)}_{2}\cup \mathbb{S}^{(3)}_{1}$, we have $c(v_5v_6v_7)\in \{1,2\}$.
Then one of $\{v_4v_1v_2, v_1v_2v_5, v_5v_6v_7\}$ and $\{v_3v_1v_2, v_1v_2v_5, v_5v_6v_7\}$ is a rainbow $\mathcal{M}$, a contradiction.
Thus $G$ contains no rainbow $\mathbb{S}^{(3)}_{3}$.
\end{proof}

By Observation~\ref{obs:S2}, we may assume that $\{v_1v_2v_3, v_1v_2v_4\}$ is a rainbow $\mathbb{S}^{(3)}_{2}$ in $G$, where $c(v_1v_2v_3)=1$ and $c(v_1v_2v_4)=2$.
By Claim~\ref{cl:messy-1}, $G$ contains no rainbow $\mathbb{S}^{(3)}_{2}\cup \mathbb{S}^{(3)}_{1}$, so $C(V(G)\setminus \{v_1, v_2, v_3, v_4\})\subseteq \{1,2\}$.
Hence, every edge of a color in $\{3, \ldots, |C(G)|\}$ must contain a vertex in $\{v_1, v_2, v_3, v_4\}$.

\begin{claim}\label{cl:messy-2}
For any edge $e$ with $c(e)\notin \{1,2\}$, we have $|e\cap \{v_1, v_2, v_3, v_4\}|=2$.
\end{claim}

\begin{proof}
Let $X=e\cap \{v_1, v_2, v_3, v_4\}$.
Then $X\neq \emptyset$ by the preceding argument.
If $|X|=1$, then $X=\{v_1\}$ or $X=\{v_2\}$ to avoid a rainbow $\mathcal{M}$.
Without loss of generality, we may assume that $e=v_1v_5v_6$.
Since $c(v_5v_6v_7)\in \{1,2\}$ and $c(e)\notin \{1,2\}$, one of $\{v_7v_6v_5, v_6v_5v_1, v_1v_2v_4\}$ and $\{v_7v_6v_5, v_6v_5v_1, v_1v_2v_3\}$ is a rainbow $\mathcal{M}$, a contradiction.

If $|X|=3$, then $e=v_1v_3v_4$ or $e=v_2v_3v_4$, say $e=v_1v_3v_4$.
We now consider the edge $v_2v_5v_6$.
If $c(v_2v_5v_6)\neq c(v_1v_3v_4)$, then one of $\{v_3v_1v_4, v_1v_4v_2, v_2v_5v_6\}$ and $\{v_4v_3v_1, v_3v_1v_2, v_2v_5v_6\}$ is a rainbow $\mathcal{M}$, a contradiction.
Thus $c(v_2v_5v_6)= c(v_1v_3v_4)$, so $c(v_2v_5v_6)\notin \{1,2\}$.
But this reduces to the case $|X|=1$ since $|\{v_2, v_5, v_6\}\cap \{v_1, v_2, v_3, v_4\}|=1$, which contradicts the preceding argument.
Therefore, we have $|X|=2$.
\end{proof}

By Claim~\ref{cl:messy-1}, $G$ contains no rainbow $\mathbb{S}^{(3)}_{3}$, so $c(v_1v_2v_i)\in \{1,2\}$ for any $i\in \{5, 6, \ldots, n\}$.
Then $c(v_3v_4v_i)\in \{1,2\}$ for any $i\in \{5, 6, \ldots, n\}$;
otherwise, if there exists $i\in \{5, 6, \ldots, n\}$ such that $c(v_3v_4v_i)\notin \{1,2\}$, then for any $j\in \{5, 6, \ldots, n\}\setminus \{i\}$, at least one of $\{v_jv_1v_2, v_1v_2v_4, v_4v_iv_3\}$ and $\{v_jv_1v_2, v_1v_2v_3, v_3v_iv_4\}$ is a rainbow $\mathcal{M}$.
Combining with Claim~\ref{cl:messy-2}, we know that for every $e$ with $c(e)\notin \{1,2\}$, $e$ is of the form $v_{\ell_1}v_{\ell_2}v_{\ell_3}$ with $\ell_1\in \{1,2\}$, $\ell_2\in \{3,4\}$ and $\ell_3 \in \{5, 6, \ldots, n\}$.
Without loss of generality, we may assume that $v_1v_3v_5$ is an edge of color 3.
By Claim~\ref{cl:messy-2}, we have $c(v_3v_5v_6)\in \{1,2\}$ and $c(v_4v_6v_7)\in \{1,2\}$.
Since $G$ contains no rainbow $\mathcal{M}$, we further have $c(v_3v_5v_6)=2$ (otherwise $\{v_6v_5v_3, v_5v_3v_1, v_1v_2v_4\}$ is a rainbow $\mathcal{M}$), and then $c(v_4v_6v_7)=2$ (otherwise $\{v_1v_3v_5, v_3v_5v_6, v_6v_7v_4\}$ is a rainbow $\mathcal{M}$).
But now $\{v_1v_3v_5, v_1v_3v_2, v_4v_6v_7\}$ is a rainbow $\mathbb{S}^{(3)}_{2}\cup \mathbb{S}^{(3)}_{1}$, contradicting Claim~\ref{cl:messy-1}.
This completes the proof of Theorem~\ref{th:messy}.
\hfill$\square$

\section{Loose path--Proof of Theorem~\ref{th:loose}}
\label{sec:loose}

In this section, we prove the structural result (Theorem~\ref{th:loose}) for the loose path $\mathcal{L}$.
We will also prove a stronger result (see Theorem~\ref{th:loose+}).
We first prove the following result on rainbow $\mathcal{L}$-free edge-colorings.

\begin{lemma}\label{le:loose-1}
For $n\geq 7$, let $G$ be a rainbow $\mathcal{L}$-free edge-coloring of $K_{n}^{(3)}$.
Then $G$ contains no rainbow $C^{(3)}_3$, $S^{(3)}_3$ or $S^{(3)}_2\cup S^{(3)}_1$.
\end{lemma}

\begin{proof}
We first show that $G$ contains no rainbow $C^{(3)}_3$.
Suppose for a contradiction that $\{v_1v_2v_3, v_3v_4v_5, v_5v_6v_1\}$ is a rainbow $C^{(3)}_3$, where $c(v_1v_2v_3)=1$, $c(v_3v_4v_5)=2$ and $c(v_5v_6v_1)=3$.
Let $v_7\in V(G)\setminus \{v_1, v_2, \ldots, v_6\}$.
In order to avoid a rainbow $\mathcal{L}$, we have $c(v_2v_7v_4)=3$, $c(v_4v_7v_6)=1$ and $c(v_6v_7v_2)=2$.
Then $c(v_1v_2v_6)=2$, since otherwise at least one of $\{v_6v_1v_2, v_2v_7v_4, v_4v_3v_5\}$ and $\{v_2v_1v_6, v_6v_7v_4, v_4v_3v_5\}$ is a rainbow $\mathcal{L}$.
By symmetry, we also have $c(v_3v_2v_4)=3$ and $c(v_5v_4v_6)=1$.
Then no matter what color is assigned to the edge $v_1v_3v_5$, at least one of $\{v_1v_5v_3, v_3v_4v_2, v_2v_7v_6\}$, $\{v_1v_3v_5, v_5v_6v_4, v_4v_7v_2\}$ and $\{v_3v_5v_1, v_1v_2v_6, v_6v_7v_4\}$ is a rainbow $\mathcal{L}$, a contradiction.
Hence, $G$ contains no rainbow $C^{(3)}_3$.

We next show that $G$ contains no rainbow $S^{(3)}_3$.
Suppose that $G$ contains a rainbow $S^{(3)}_3$, say $\{v_1v_2v_3, v_1v_4v_5, v_1v_6v_7\}$.
Then no matter what color is assigned to $v_2v_4v_6$, at least one of $\{v_1v_3v_2, v_2v_6v_4, v_4v_5v_1\}$, $\{v_1v_5v_4, v_4v_2v_6, v_6v_7v_1\}$ and $\{v_1v_3v_2, v_2v_4v_6, v_6v_7v_1\}$ is a rainbow $C^{(3)}_3$, which contradicts the preceding argument.
Hence, $G$ contains no rainbow $S^{(3)}_3$.

Finally, we show that $G$ contains no rainbow $S^{(3)}_2\cup S^{(3)}_1$.
This is trivial when $n=7$ since $\big|V\big(S^{(3)}_2\cup S^{(3)}_1\big)\big|=8$, so we may assume that $n\geq 8$.
Suppose for a contradiction that $\{v_1v_2v_3, v_1v_4v_5, v_6v_7v_8\}$ is a rainbow $S^{(3)}_2\cup S^{(3)}_1$, where $c(v_1v_2v_3)=1$, $c(v_1v_4v_5)=2$ and $c(v_6v_7v_8)=3$.
Since $G$ contains no rainbow $C^{(3)}_3$, we have $c(v_3v_6v_5)\in \{1,2\}$.
Then at least one of $\{v_1v_4v_5, v_5v_3v_6, v_6v_7v_8\}$ and $\{v_1v_2v_3, v_3v_5v_6, v_6v_7v_8\}$ is a rainbow $\mathcal{L}$, a contradiction.
Hence, $G$ contains no rainbow $S^{(3)}_2\cup S^{(3)}_1$.
\end{proof}

Applying our result (Theorem~\ref{th:messy}) for the messy path $\mathcal{M}$, we can prove the following result, which already improves Theorem~\ref{th:Liu-stru}.

\begin{lemma}\label{le:loose-2}
For any integer $n\geq 7$, let $G$ be a rainbow $\mathcal{L}$-free edge-colored $K^{(3)}_n$ with $|C(G)|\geq 3$.
Then there exist two vertices $u, v$ such that $G- \{u,v\}$ is monochromatic.
\end{lemma}

\begin{proof}
Since $|C(G)|\geq 3$, $G$ contains a rainbow $\mathcal{M}$ by Theorem~\ref{th:messy}.
Without loss of generality, we may assume that $\{v_1v_2v_3, v_2v_3v_4, v_4v_5v_6\}$ is a rainbow $\mathcal{M}$, where $c(v_1v_2v_3)=1$, $c(v_2v_3v_4)=2$ and $c(v_4v_5v_6)=3$.
Let $U=V(G)\setminus \{v_1, v_2, \ldots, v_6\}$ and $W=U\cup \{v_5, v_6\}$.

\begin{claim}\label{cl:loose-le-1}
If $|U|\geq 3$, then $C(U)=\{3\}$.
\end{claim}

\begin{proof}
Note that $\{v_2v_3v_4, v_4v_5v_6\}$ is a rainbow $S^{(3)}_2$ using colors 2 and 3.
By Lemma~\ref{le:loose-1}, $G$ contains no rainbow $S^{(3)}_2\cup S^{(3)}_1$, so $C(U)\subseteq \{2,3\}$ when $|U|\geq 3$.
Suppose that $c(v_iv_jv_{\ell})=2$ for some $v_i, v_j, v_{\ell}\in U$.
Then no matter what color is assigned to $v_3v_4v_i$, at least one of $\{v_jv_{\ell}v_i, v_iv_3v_4, v_4v_5v_6\}$, $\{v_1v_2v_3, v_3v_iv_4, v_4v_5v_6\}$ and $\{v_1v_2v_3, v_3v_4v_i, v_iv_jv_{\ell}\}$ is a rainbow $\mathcal{L}$, a contradiction.
\end{proof}

\begin{claim}\label{cl:loose-le-2}
$C(W)=\{3\}$.
\end{claim}

\begin{proof}
By Claim~\ref{cl:loose-le-1}, it suffices to consider edges containing $v_5$ or $v_6$.
By symmetry, we only consider $v_5v_iv_j$ (when $|U|\geq 2$) and $v_5v_6v_i$, where $v_i, v_j\in U$.
In order to avoid a rainbow $\mathcal{L}$, we have $c(v_5v_iv_j)\in \{2,3\}$.
If $c(v_5v_iv_j)=2$, then $\{v_5v_iv_j, v_5v_4v_6, v_1v_2v_3\}$ is a rainbow $S^{(3)}_2\cup S^{(3)}_1$, contradicting Lemma~\ref{le:loose-1}.
Thus $c(v_5v_iv_j)=3$.
Next, we consider $v_5v_6v_i$.
By Lemma~\ref{le:loose-1}, $G$ contains no rainbow $S^{(3)}_3$, so $c(v_4v_iv_1)\in \{2, 3\}$ since otherwise $\{v_4v_2v_3, v_4v_5v_6, v_4v_iv_1\}$ is a rainbow $S^{(3)}_3$.
Moreover, we have $c(v_4v_iv_1)\in \{1, 3\}$ since otherwise $\{v_3v_2v_1, v_1v_iv_4, v_4v_5v_6\}$ is a rainbow $\mathcal{L}$.
Hence, $c(v_4v_iv_1)=3$.
Then $c(v_5v_6v_i)=3$, since otherwise at least one of $\{v_5v_6v_i, v_iv_1v_4, v_4v_3v_2\}$ and $\{v_5v_6v_i, v_iv_4v_1, v_1v_2v_3\}$ is a rainbow $\mathcal{L}$, a contradiction.
\end{proof}

Now we have $c(v_1v_2v_3)=1$, $c(v_2v_3v_4)=2$ and $C(W)=\{3\}$.
Note that under these conditions, the vertices $v_1$ and $v_4$ are symmetric.
In the remainder of the proof, we show that every edge whose color is not 3 must contain at least one of $v_2$ and $v_3$.
This implies that $G-\{v_2, v_3\}$ is monochromatic in color 3, and thus completes the proof.
Suppose that there exists an edge $e$ with $c(e)\neq 3$ and $e\cap \{v_2, v_3\}=\emptyset$.
Since $C(W)=\{3\}$, we have $e\cap \{v_1, v_4\}\neq \emptyset$.
By symmetry, it suffices to consider $v_1v_4v_i$ and $v_1v_iv_j$, where $v_i, v_j\in W$.
Since $C(W)=\{3\}$ and $|W|=n-4\geq 3$, we have $c(v_1v_4v_i)=3$ for all $v_i\in W$; otherwise the edge $v_1v_4v_i$, an edge containing $v$ in $W$, and one of $v_1v_2v_3$ and $v_2v_3v_4$ form a rainbow $\mathcal{L}$.
Then we further have $c(v_1v_iv_j)=3$.
To see this, choose a vertex $v_{\ell}\in W\setminus \{v_i, v_j\}$.
If $c(v_1v_iv_j)\notin \{2,3\}$, then $\{v_2v_3v_4, v_4v_{\ell}v_1, v_1v_iv_j\}$ is a rainbow $\mathcal{L}$, a contradiction;
if $c(v_1v_iv_j)=2$, then $\{v_1v_2v_3, v_1v_{\ell}v_4, v_1v_iv_j\}$ is a rainbow $S^{(3)}_3$, contradicting Lemma~\ref{le:loose-1}.
Hence, we have $c(v_1v_iv_j)=3$ for any $v_i, v_j\in W$.
This completes the proof of Lemma~\ref{le:loose-2}.
\end{proof}

We next state and prove two technical lemmas.

\begin{lemma}\label{le:loose-3}
For any integer $n\geq 7$, let $G$ be a rainbow $\mathcal{L}$-free edge-colored $K^{(3)}_n$ with $|C(G)|\geq 3$.
Suppose that there exists a subset $U\subseteq V(G)$ with $|U|\leq 2$ such that $G- U$ is monochromatic, say $C(G- U)=\{1\}$.
Then for any two edges $e_1, e_2 \in E(G)$ with $c(e_1)\neq 1$, $c(e_2)\neq 1$ and $c(e_1)\neq c(e_2)$, we have $|e_1\cap e_2|=2$.
\end{lemma}

\begin{proof}
By the assumption, we can find two vertices $u,v\in V(G)$ such that $C(G-\{u,v\})=\{1\}$ (since $|C(G)|\geq 3$, we have $U\neq \emptyset$; if $|U|=1$, enlarge $U$ by one arbitrary vertex).
Since $n\geq 7$, we have $|V(G)\setminus \{u,v\}|\geq 5$.
Since $c(e_1)\neq 1$ and $c(e_2)\neq 1$, we have $e_1\cap \{u,v\}\neq \emptyset$ and $e_2\cap \{u,v\}\neq \emptyset$.

If $e_1\cap e_2 =\emptyset$, then we may assume that $e_1=uw_1w_2$ and $e_2=vw_3w_4$, where $w_1, w_2, w_3, w_4 \in V(G)\setminus \{u,v\}$.
Since $C(G- \{u,v\})=\{1\}$ and $c(e_1)\neq c(e_2)$, we have that $\{uw_1w_2, w_2w_5w_4, w_4w_3v\}$ is a rainbow $\mathcal{L}$ for some $w_5\in V(G)\setminus \{u, v, w_1, w_2, w_3, w_4\}$, a contradiction.

If $|e_1\cap e_2| =1$, then $\{e_1, e_2\}$ is a rainbow $S^{(3)}_2$.
Denote the vertex in $e_1\cap e_2$ by $x$, and let $y,z$ be two distinct vertices in $V(G)\setminus (e_1\cup e_2)$.
By Lemma~\ref{le:loose-1}, $G$ contains no rainbow $S^{(3)}_3$, so $c(xyz)\in \{c(e_1), c(e_2)\}$, say $c(xyz)=c(e_1)$.
Then $c(xyz)\neq 1$, and thus $\{x,y,z\}\cap \{u,v\}\neq \emptyset$.
Combining the properties $e_1\cap e_2=\{x\}$, $e_1\cap \{u,v\}\neq \emptyset$, $e_2\cap \{u,v\}\neq \emptyset$, $\{x,y,z\}\cap \{u,v\}\neq \emptyset$ and $y,z\notin e_1\cup e_2$, we must have $x\in \{u,v\}$.
Now $(e_1\setminus \{x\})\setminus \{u,v\}\neq \emptyset$, $(e_2\setminus \{x\})\setminus \{u,v\}\neq \emptyset$ and $\{y,z\}\setminus \{u,v\}\neq \emptyset$, say $a_1\in (e_1\setminus \{x\})\setminus \{u,v\}$, $a_2\in (e_2\setminus \{x\})\setminus \{u,v\}$ and $y\in \{y,z\}\setminus \{u,v\}$.
Then $a_1, a_2, y\in V(G)\setminus \{u,v\}$, so $c(a_1a_2y)=1$.
Now $\{e_2, a_2a_1y, yzx\}$ is a rainbow $C^{(3)}_3$, contradicting Lemma~\ref{le:loose-1}.
Therefore, we have $|e_1\cap e_2|=2$.
\end{proof}

\begin{lemma}\label{le:loose-4}
For any integer $n\geq 7$, let $G$ be a rainbow $\mathcal{L}$-free edge-colored $K^{(3)}_n$ with $|C(G)|\geq 3$.
Suppose that neither conclusion~{\rm(i)} nor conclusion~{\rm(ii)} of Theorem~\ref{th:loose} holds, and there exist two vertices $u, v$ such that $G- \{u,v\}$ is monochromatic, say $C(G- \{u,v\})=\{1\}$.
Then for any two edges $e_1, e_2 \in E(G)$ with $c(e_1)\neq 1$, $c(e_2)\neq 1$ and $c(e_1)\neq c(e_2)$, at least one of $e_1$ and $e_2$ contains both $u$ and $v$.
\end{lemma}

\begin{proof}
Let $V(G)=\{u, v\}\cup U$, where $U=\{w_1, w_2, \ldots, w_{n-2}\}$.
Since $n\geq 7$, we have $|U|\geq 5$.
By the assumption, we have $e_1\cap \{u, v\}\neq \emptyset$ and $e_2\cap \{u, v\}\neq \emptyset$.
We first show that $e_1\cup e_2$ contains both $u$ and $v$.

\begin{claim}\label{cl:loose-le-4}
For any two edges $e_1, e_2 \in E(G)$ with $c(e_1)\neq 1$, $c(e_2)\neq 1$ and $c(e_1)\neq c(e_2)$, we have $\{u,v\}\subseteq e_1\cup e_2$.
\end{claim}

\begin{proof}
Since $e_1\cap \{u, v\}\neq \emptyset$ and $e_2\cap \{u, v\}\neq \emptyset$, we have $|(e_1\cup e_2)\cap \{u, v\}|\geq 1$.
Suppose that $|(e_1\cup e_2)\cap \{u, v\}|=1$.
Combining with $|e_1\cap e_2|=2$ (which follows from Lemma~\ref{le:loose-3}), we may assume that $e_1=uw_1w_2$, $e_2=uw_1w_3$, $c(e_1)=2$ and $c(e_2)=3$.
Since conclusion~(i) of Theorem~\ref{th:loose} does not hold, there exists an edge $f\in E(G-u)$ with $c(f)\neq 1$.
Then $v\in f$, and by Lemma~\ref{le:loose-3}, either $c(f)=2$ and $f=vw_1w_3$, or $c(f)=3$ and $f=vw_1w_2$.
Without loss of generality, we may assume that $c(f)=2$ and $f=vw_1w_3$.

By Lemma~\ref{le:loose-3}, every edge with a color in $C(G)\setminus \{1,3\}$ must contain two vertices of $e_2$,
and every edge of color 3 must contain two vertices of $e_1$ and two vertices of $f$.
Combining with $C(G- \{u,v\})=\{1\}$, only edges $w_1uv$ and $w_1vw_2$ can be of color 3 in $E(G)\setminus \{e_2\}$.
Moreover, there exists an edge of some color in $C(G)\setminus \{1\}$ that contains at most one vertex of $e_2$; otherwise conclusion~(ii) of Theorem~\ref{th:loose} holds (in which we take $e_2$ to be $e$ and color 1 to be color $i$).
From the above arguments, such an edge can only be $w_1vw_2$, and the edge $w_1vw_2$ must be of color 3.
Furthermore, since conclusion~(i) of Theorem~\ref{th:loose} does not hold, there exists an edge $g\in E(G-w_1)$ with $c(g)\neq 1$.
In particular, we have $g\notin \{w_1uv, w_1vw_2\}$.
Combining with the above arguments, we have $c(g)\neq 3$ and $g\cap e_2=\{u, w_3\}$.
Now $|g\cap \{w_1, v, w_2\}|\leq 1$.
Since $c(w_1vw_2)=3$ and $c(g)\notin \{1,3\}$, we can deduce a contradiction by Lemma~\ref{le:loose-3}.
The proof of Claim~\ref{cl:loose-le-4} is complete.
\end{proof}

We now show that at least one of $e_1$ and $e_2$ contains both $u$ and $v$.
For a contradiction, suppose that $|e_1\cap \{u,v\}|=1$ and $|e_2\cap \{u,v\}|=1$.
By Claim~\ref{cl:loose-le-4}, we may assume that $u\in e_1$ and $v\in e_2$.
By Lemma~\ref{le:loose-3}, we have $|e_1\cap e_2|=2$, say $e_1=uw_1w_2$ and $e_2=vw_1w_2$.
Without loss of generality, we may assume that $c(e_1)=2$ and $c(e_2)=3$.

Moreover, there exists an edge $f$ with $c(f)\neq 1$ and $f\setminus \{u, v, w_1, w_2\}\neq \emptyset$; otherwise, if every edge with a color in $C(G)\setminus \{1\}$ is contained in $\{u, v, w_1, w_2\}$, then conclusion~(ii) of Theorem~\ref{th:loose} holds (in which we take $e_1$ to be $e$ and color 1 to be color $i$).
Since $c(f)\neq 1$, we have $f\cap \{u, v\}\neq \emptyset$, say $u\in f$.
By Lemma~\ref{le:loose-3}, at least one of $|f\cap e_1|=2$ and $|f\cap e_2|=2$ holds.
Thus $f\cap \{w_1, w_2\}\neq \emptyset$, say $w_1\in f$.
Now $f\cap \{u, v, w_1, w_2\}=\{u, w_1\}$.
In particular, $|f\cap e_2|=1$.
Thus $c(f)=c(e_2)=3$ by Lemma~\ref{le:loose-3}.
But then $c(f)\neq c(e_1)$ and $(f\cup e_1)\cap \{u,v\}=\{u\}$, contradicting Claim~\ref{cl:loose-le-4}.
This completes the proof of Lemma~\ref{le:loose-4}.
\end{proof}

Now we have all ingredients to present our proof of Theorem~\ref{th:loose}.

\vspace{0.3cm}
\noindent{\bf Proof of Theorem~\ref{th:loose}.}~
Let $G$ be a rainbow $\mathcal{L}$-free edge-colored $K^{(3)}_n$ with $|C(G)|\geq 3$, and suppose for a contradiction that Theorem~\ref{th:loose}~$($i$)$ and $($ii$)$ do not hold.
By Lemma~\ref{le:loose-2}, there exist two vertices $u, v$ such that $G- \{u,v\}$ is monochromatic, say $C(G-\{u,v\})=\{1\}$.
Since (i) does not hold, there exists an edge $e\in E(G-u)$ with $c(e)\neq 1$.
Then $v\in e$ since $C(G-\{u,v\})=\{1\}$, say $e=vw_1w_2$ and $c(e)=2$.
Let $f\in E(G)$ be an edge of color 3.
By Lemma~\ref{le:loose-3}, we have $|f\cap e|=2$.
Since $|e\cap \{u,v\}|=1$, Lemma~\ref{le:loose-4} then implies $\{u,v\}\subseteq f$.
Thus $f$ is $uvw_1$ or $uvw_2$, say $f=uvw_1$.
This also implies that the only possible edge of color 3 in $E(G)\setminus \{f\}$ is $uvw_2$.
Note that $|\{u,v,w_2\}\cap f|=2$.
By Lemma~\ref{le:loose-3}, we can also deduce that every edge $g\in E(G)$ with $c(g)\notin \{1,3\}$ satisfies $|g\cap f|=2$.
Therefore, every edge $h\in E(G)\setminus \{f\}$ with $c(h)\neq 1$ satisfies $|h\cap f|=2$, so Theorem~\ref{th:loose}~(ii) holds with distinguished edge $f$ and distinguished color $i=1$.
This completes the proof of Theorem~\ref{th:loose}.
\hfill$\square$
\vspace{0.3cm}

Finally, we prove the following stronger result.
This result will be used in our proof of Theorem~\ref{th:anti-Ram}~(iii).

\begin{theorem}\label{th:loose+}
For any integer $n\geq 7$, let $G$ be a rainbow $\mathcal{L}$-free edge-colored $K^{(3)}_n$ with $|C(G)|\geq 3$.
Then at least one of the following statements holds $($after renumbering the colors if necessary$)$:
\begin{itemize}
\item[{\rm (i)}] there exist two vertices $u,v$ such that every edge $e\in E(G)$ with $c(e)\neq 1$ satisfies $\{u,v\}\subseteq e$;

\item[{\rm (ii)}] $|C(G)|\in \{3, 4, 5\}$, and there exists a vertex $v$ such that all but at most one edge of $G-v$ are of the same color;

\item[{\rm (iii)}] $|C(G)|=3$, and there exists an edge $e$ with $c(e)\neq 1$ such that every edge $f\in E(G)\setminus \{e\}$ with $c(f)\neq 1$ satisfies $|e\cap f|=2$.
\end{itemize}
\end{theorem}

\begin{proof}
Let $V(G)=\{v_1, v_2, \ldots, v_n\}$.
When $|C(G)|=3$, if Theorem~\ref{th:loose}~(i) holds, then Theorem~\ref{th:loose+}~(ii) holds with no exceptional edge; if Theorem~\ref{th:loose}~(ii) holds, then Theorem~\ref{th:loose+}~(iii) holds after renumbering the colors.
Hence, we may assume that $|C(G)|\geq 4$ in the following arguments.
By Lemma~\ref{le:loose-2}, there exists a subset $U\subseteq V(G)$ with $|U|\leq 2$ such that $G- U$ is monochromatic, say $C(G- U)=\{1\}$.

\begin{claim}\label{cl:loose-le-5}
If there exist three edges $e_1, e_2, e_3$ with $|e_1\cup e_2 \cup e_3|=4$ such that $c(e_1)$, $c(e_2)$, $c(e_3)$ are three distinct colors in $\{2, 3, \ldots, |C(G)|\}$, then conclusion~{\rm(ii)} holds.
\end{claim}

\begin{proof}
Without loss of generality, suppose that $c(v_1v_2v_3)=2$, $c(v_1v_2v_4)=3$ and $c(v_1v_3v_4)=4$.
By Lemma~\ref{le:loose-3}, an edge of a color in $\{5, \ldots, |C(G)|\}$ must intersect each of $v_1v_2v_3$, $v_1v_2v_4$, $v_1v_3v_4$ in exactly two vertices.
The only possible edge is $v_2v_3v_4$, so $|C(G)|\leq 5$.
Moreover, by Lemma~\ref{le:loose-3}, every edge $g\in E(G)\setminus \{v_1v_2v_3, v_1v_2v_4, v_1v_3v_4, v_2v_3v_4\}$ with $c(g)\in \{2, 3, \ldots, |C(G)|\}$ must contain $v_1$.
This implies that conclusion~(ii) holds (in which we take $v_1$ to be $v$ and $v_2v_3v_4$ to be the exceptional edge).
\end{proof}

Let $e$ and $f$ be two edges of colors 2 and 3, respectively.
By Lemma~\ref{le:loose-3}, we have $|e\cap f|=2$, so we may assume that $e=v_1v_2v_3$ and $f=v_1v_2v_4$.
We first assume that $c(v_1v_2v_i)\in \{1, 2, 3\}$ for all $i\in \{5, 6, \ldots, n\}$.
Now every edge of a color in $\{4, \ldots, |C(G)|\}$ contains at most one of $v_1$ and $v_2$.
Moreover, such an edge must intersect both $e$ and $f$ in two vertices by Lemma~\ref{le:loose-3}.
Hence, an edge of a color in $\{4, \ldots, |C(G)|\}$ can only be $v_1v_3v_4$ or $v_2v_3v_4$.
Since $|C(G)|\geq 4$, the edges $e$, $f$ and one of $v_1v_3v_4$ and $v_2v_3v_4$ form three edges satisfying the condition of Claim~\ref{cl:loose-le-5}, so conclusion~$($ii$)$ holds.

Now we assume that $c(v_1v_2v_i)\notin \{1, 2, 3\}$ for some $i\in \{5, 6, \ldots, n\}$, say $c(v_1v_2v_5)=4$.
For any $j\in [2]$, if one of $c(v_jv_3v_4)=4$, $c(v_jv_3v_5)=3$ or $c(v_jv_4v_5)=2$ holds, then there exist three edges satisfying the condition of Claim~\ref{cl:loose-le-5}, so conclusion~$($ii$)$ holds.
Thus we may assume that for each $j\in [2]$, none of $c(v_jv_3v_4)=4$, $c(v_jv_3v_5)=3$ and $c(v_jv_4v_5)=2$ holds.
We claim that every edge of a color in $\{2, 3, \ldots, |C(G)|\}$ must contain both $v_1$ and $v_2$.
Indeed, every edge of color 2 must intersect both $v_1v_2v_4$ and $v_1v_2v_5$ in two vertices by Lemma~\ref{le:loose-3}.
If such an edge does not contain both $v_1$ and $v_2$, then it can only be $v_1v_4v_5$ or $v_2v_4v_5$, which was excluded.
Similarly, if an edge of color 3 (resp., color 4) does not contain both $v_1$ and $v_2$, then it can only be $v_1v_3v_5$ or $v_2v_3v_5$ (resp., $v_1v_3v_4$ or $v_2v_3v_4$), which was excluded.
Moreover, an edge of a color in $\{5, \ldots, |C(G)|\}$ must intersect each of $v_1v_2v_3$, $v_1v_2v_4$ and $v_1v_2v_5$ in two vertices by Lemma~\ref{le:loose-3}, so it must contain both $v_1$ and $v_2$.
Therefore, every edge of a color in $\{2, 3, \ldots, |C(G)|\}$ must contain both $v_1$ and $v_2$, so conclusion~(i) holds.
The proof of Theorem~\ref{th:loose+} is complete.
\end{proof}

\section{Proofs of Ramsey-type results}
\label{sec:Ramsey}

In this section, we prove the Ramsey-type results (Theorems~\ref{th:Ram-tight}, \ref{th:Ram-messy}, \ref{th:Ram-loose} and \ref{th:anti-Ram}) by applying our structural results (Theorems~\ref{th:tight}, \ref{th:messy}, \ref{th:loose} and \ref{th:loose+}).

\vspace{0.3cm}
\noindent{\bf Proof of Theorem~\ref{th:Ram-tight}.}~
Since $\mathcal{T}$ has three edges, there is no rainbow $\mathcal{T}$ in a 2-edge-colored $K^{(3)}_{n}$.
Hence, we have $f(H, \mathcal{T})\geq R_2(H)$.
Next, we show that $f(H, \mathcal{T})\leq R_2(H)$ for every connected 3-graph $H$ with $R_2(H)\geq 5$.
Let $G$ be an edge-colored $K^{(3)}_{n}$ with $n=R_2(H)\geq 5$.
Suppose that $G$ contains no rainbow $\mathcal{T}$, and we shall show that $G$ contains a monochromatic $H$.
If $|C(G)|\leq 2$, then $G$ contains a monochromatic $H$ since $n= R_2(H)$.
Thus we may assume that $|C(G)|\geq 3$.
By Theorem~\ref{th:tight}, we can partition $V(G)$ into $|C(G)|-1$ parts $V_2, V_3, \ldots, V_{|C(G)|}$ such that $\{i\}\subseteq C(V_i)\subseteq \{1,i\}$ for every $i\in \{2, 3, \ldots, |C(G)|\}$, and all the remaining edges are of color 1.
Let $G'$ be an auxiliary 2-edge-colored $K^{(3)}_n$ obtained from $G$ by recoloring all edges of colors in $\{3, \ldots, |C(G)|\}$ with color 2.
Note that every edge of color 2 in $G'$ lies entirely in one of $V_2, V_3, \ldots, V_{|C(G)|}$.
Since $n=R_2(H)$, $G'$ contains a monochromatic $H$.
If this monochromatic $H$ has color 1, then $G$ also contains a monochromatic $H$ of color 1.
If this monochromatic $H$ has color 2, then since $H$ is connected, there exists an $i_0\in \{2, \ldots, |C(G)|\}$ such that $H$ is contained in $V_{i_0}$.
This implies that $G$ contains a monochromatic $H$ of color $i_0$.
In either case, $G$ contains a monochromatic $H$, which implies that $f(H, \mathcal{T})\leq R_2(H)$.
Therefore, we have $f(H, \mathcal{T})= R_2(H)$.
\hfill$\square$

\vspace{0.3cm}
\noindent{\bf Proof of Theorem~\ref{th:Ram-messy}.}~
Since $\mathcal{M}$ has three edges, there is no rainbow $\mathcal{M}$ in a 2-edge-colored $K^{(3)}_{n}$.
Hence, we have $f(H, \mathcal{M})\geq R_2(H)$.
On the other hand, if $G$ is an edge-colored $K^{(3)}_n$ with $n=R_2(H)\geq 7$, then either $G$ contains a rainbow $\mathcal{M}$ or $|C(G)|\leq 2$ by Theorem~\ref{th:messy}.
In the latter case, $G$ contains a monochromatic $H$ since $n= R_2(H)$.
Thus $f(H, \mathcal{M})\leq R_2(H)$, and therefore $f(H, \mathcal{M})= R_2(H)$.
\hfill$\square$

\vspace{0.3cm}
\noindent{\bf Proof of Theorem~\ref{th:Ram-loose}.}~
Since $\mathcal{L}$ has three edges, there is no rainbow $\mathcal{L}$ in a 2-edge-colored $K^{(3)}_{n}$.
Hence, we have $f(H, \mathcal{L})\geq R_2(H)$.
Next, we show that $f(H, \mathcal{L})\leq R_2(H)$ for every 3-graph $H$ with $R_2(H)\geq \max\{|V(H)|+1, 7\}$.
Let $G$ be an edge-colored $K^{(3)}_{n}$ with $n=R_2(H)$ and $V(G)=\{v_1, v_2, \ldots, v_n\}$.
Suppose that $G$ contains no rainbow $\mathcal{L}$, and we shall show that $G$ contains a monochromatic $H$.
If $|C(G)|\leq 2$, then $G$ contains a monochromatic $H$ since $n= R_2(H)$.
Thus we may assume that $|C(G)|\geq 3$.
By Theorem~\ref{th:loose}, at least one of the following statements holds:
\begin{itemize}
\item[{\rm (i)}] there exists a vertex $u\in V(G)$ such that $G-u$ is monochromatic;

\item[{\rm (ii)}] there exists an edge $e\in E(G)$ and a color $i\in C(G)$ with $c(e)\neq i$ such that every edge $f\in E(G)\setminus \{e\}$ with $c(f)\neq i$ satisfies $|f\cap e|= 2$.
\end{itemize}

If (i) holds, then $G$ contains a monochromatic $K^{(3)}_{n-1}$, which contains a monochromatic $H$ since $n-1\geq |V(H)|$.
In the following argument, we assume that (ii) holds.
After renumbering the colors, let $e=v_1v_2v_3$ and $i=1$.
Let $F_1$ be the spanning subgraph of $G$ consisting of all edges of colors in $\{2, \ldots, |C(G)|\}$,
and let $F_2$ be the spanning subgraph of $G$ consisting of all edges of color 1.
We next show that $F_2$ contains a subgraph that is isomorphic to $F_1$ (here $F_1$ and $F_2$ are uncolored subgraphs).
To see this, let $F_1'$ be the spanning subgraph of $G$ with edge set $\{v_jv_kv_{\ell}\colon\, 1\leq j<k\leq 3, \ell\in [n]\setminus [3]\}\cup \{v_1v_2v_3\}$,
and let $F_2'$ be the spanning subgraph of $G$ with edge set $\{v_jv_kv_{\ell}\colon\, 4\leq j<k\leq 6, \ell\in [n]\setminus \{4,5,6\}\}\cup \{v_4v_5v_6\}$.
Then by~(ii), $F_1\subseteq F_1'$, $F_2'\subseteq F_2$, and $F_1'$ is isomorphic to $F_2'$.
Thus $F_2$ contains a subgraph that is isomorphic to $F_1$.
Let $G'$ be an auxiliary 2-edge-colored $K^{(3)}_n$ obtained from $G$ by recoloring all edges in $F_1$ with color 2.
Since $n= R_2(H)$, $G'$ contains a monochromatic $H$.
If this monochromatic $H$ has color 1, then $G$ also contains a monochromatic $H$ of color 1.
If this monochromatic $H$ has color 2, then it lies entirely in $F_1$.
Since $F_2$ contains a subgraph that is isomorphic to $F_1$, $G$ also contains a monochromatic $H$ of color 1 that lies entirely in $F_2$.
Thus $G$ contains a monochromatic $H$, which implies that $f(H, \mathcal{L})\leq R_2(H)$.
Therefore, we have $f(H, \mathcal{L})= R_2(H)$.
\hfill$\square$

\vspace{0.3cm}
\noindent{\bf Proof of Theorem~\ref{th:anti-Ram}.}~
(i) For the lower bound, we construct an edge-colored $K^{(3)}_n$ as follows.
Let $\{v_1, v_2, \ldots, v_n\}$ be the vertex set, and we color the edges such that
\begin{itemize}
\item $c(v_{3i-2}v_{3i-1}v_{3i})=i$ for each $1\leq i\leq \left\lfloor\frac{n}{3}\right\rfloor$; and

\item all the remaining edges are of color $\left\lfloor\frac{n}{3}\right\rfloor+1$.
\end{itemize}
Note that the edges of colors in $\{1, 2, \ldots, \left\lfloor\frac{n}{3}\right\rfloor\}$ are pairwise disjoint.
Since $|V(\mathcal{T})|=5$, every copy of $\mathcal{T}$ in this edge-colored $K^{(3)}_n$ contains at most one edge of a color in $\{1, 2, \ldots, \left\lfloor\frac{n}{3}\right\rfloor\}$.
Thus this is a rainbow $\mathcal{T}$-free edge-coloring, so $ar(n,\mathcal{T})>\left\lfloor\frac{n}{3}\right\rfloor+1$.

For the upper bound, let $G$ be an edge-colored $K^{(3)}_n$ without rainbow $\mathcal{T}$.
We now show that $|C(G)|\leq \left\lfloor\frac{n}{3}\right\rfloor+1$.
If $|C(G)|\leq 2$, then $|C(G)|\leq 2 \leq \left\lfloor\frac{n}{3}\right\rfloor+1$ for $n\geq 5$, and we are done.
If $|C(G)|\geq 3$, then by Theorem~\ref{th:tight}, we can partition $V(G)$ into $|C(G)|-1$ parts $V_2, V_3, \ldots, V_{|C(G)|}$ such that $\{i\}\subseteq C(V_i)\subseteq \{1,i\}$ for every $i\in \{2, 3, \ldots, |C(G)|\}$, and all the remaining edges are of color 1.
For each $i\in \{2, 3, \ldots, |C(G)|\}$, since $\{i\}\subseteq C(V_i)$, we have $|V_i|\geq 3$, and thus $|C(G)|-1\leq \left\lfloor\frac{n}{3}\right\rfloor$.
Thus $|C(G)|\leq \left\lfloor\frac{n}{3}\right\rfloor+1$.
Therefore, we have $ar(n,\mathcal{T})\leq \left\lfloor\frac{n}{3}\right\rfloor+2$, and thus $ar(n,\mathcal{T})= \left\lfloor\frac{n}{3}\right\rfloor+2$.

(ii) We first prove the case $n=6$.
For the lower bound, consider the edge-colored $K_6^{(3)}$ defined in Remark~\ref{re:theorems-1}~(i),
that is, an edge-colored $K_6^{(3)}$ using 10 distinct colors such that each color class is a perfect matching.
Since the messy path $M$ contains a perfect matching, such an edge-colored $K_6^{(3)}$ is rainbow $\mathcal{M}$-free.
Thus we have $ar(6, \mathcal{M})>10$.
For the upper bound, let $G$ be an edge-colored $K^{(3)}_6$ with at least 11 colors.
Note that the 20 edges of $K^{(3)}_6$ can be decomposed into 10 perfect matchings.
Since 11 colors are used, at least one of the 10 perfect matchings must be rainbow, say $\{v_1v_2v_3, v_4v_5v_6\}$.
Let $e$ be an edge with $c(e)\notin \{c(v_1v_2v_3), c(v_4v_5v_6)\}$.
Note that we have either $|e\cap \{v_1,v_2,v_3\}|=2$ and $|e\cap \{v_4,v_5,v_6\}|=1$, or $|e\cap \{v_4,v_5,v_6\}|=2$ and $|e\cap \{v_1,v_2,v_3\}|=1$.
In both cases, there is a rainbow $\mathcal{M}$.
Therefore, we have $ar(6, \mathcal{M})\leq 11$, and thus $ar(6, \mathcal{M})= 11$.
For the case $n\geq 7$, since a 2-edge-colored hypergraph certainly contains no rainbow $\mathcal{M}$, we have $ar(n, \mathcal{M})>2$.
Moreover, if an edge-colored $K^{(3)}_n$ contains no rainbow $\mathcal{M}$, then it is colored by at most two colors by Theorem~\ref{th:messy}.
Hence, we have $ar(n, \mathcal{M})\leq 3$ for $n\geq 7$.
Therefore, we have $ar(n, \mathcal{M})= 3$ for $n\geq 7$.

(iii) For the lower bound, we construct an edge-colored $K^{(3)}_n$ as follows.
Let $\{v_1, v_2, \ldots, v_n\}$ be the vertex set, and we color the edges such that
\begin{itemize}
\item $c(v_{i}v_{n-1}v_{n})=i$ for each $i\in [n-2]$; and

\item all the remaining edges are of color $n-1$.
\end{itemize}
Note that every copy of $\mathcal{L}$ in this edge-colored $K^{(3)}_n$ contains at most one edge of a color in $[n-2]$.
Thus this is a rainbow $\mathcal{L}$-free edge-coloring, so $ar(n,\mathcal{L})>n-1$.

For the upper bound, let $G$ be an edge-colored $K^{(3)}_n$ without rainbow $\mathcal{L}$.
We shall show that $|C(G)|\leq n-1$, which implies that $ar(n, \mathcal{L})\leq n$.
If $|C(G)|\leq 2$, then $|C(G)|\leq 2 < n-1$ for $n\geq 7$, and we are done.
If $|C(G)|\geq 3$, then one of Theorem~\ref{th:loose+}~(i), (ii) and (iii) holds.
If Theorem~\ref{th:loose+}~(ii) or (iii) holds, then $|C(G)|\leq 5 < n-1$ since $n\geq 7$.
If Theorem~\ref{th:loose+}~(i) holds, then every edge whose color is not 1 has the form $uvx$ with $x\in V(G)\setminus \{u,v\}$.
Thus there can be at most $n-2$ colors other than color 1, so $|C(G)|\leq (n-2)+1=n-1$.
Hence, we have $ar(n, \mathcal{L})\leq n$ for $n\geq 7$.
Therefore, we have $ar(n, \mathcal{L})= n$ for $n\geq 7$.
\hfill$\square$

\section{Proofs of the multipartite version of the results}
\label{sec:multipartite}

In this section, we prove the multipartite version of the structural results and the Ramsey-type results introduced in Section~\ref{subsec:multipartite}.
We begin with the following observation.

\begin{observation}\label{obs:S2-multipar}
For any integer $n\geq 3$, if $K^{(3)}_{n,n,n}$ is edge-colored with at least two colors, then it contains both a rainbow $S^{(3)}_2$ and a rainbow $\mathbb{S}^{(3)}_{2}$.
\end{observation}

\begin{proof} Let $G$ be an edge-colored $K^{(3)}_{n,n,n}$ with $|C(G)|\geq 2$ and $n\geq 3$.
Let $V_1, V_2, V_3$ be the partite sets of $G$, where $V_1=\{x_1, x_2, \ldots, x_n\}$, $V_2=\{y_1, y_2, \ldots, y_n\}$ and $V_3=\{z_1, z_2, \ldots, z_n\}$.
Since $|C(G)|\geq 2$, there exist two edges $e$ and $f$ of distinct colors such that $e\cap f\neq \emptyset$.
Thus $G$ contains either a rainbow $S^{(3)}_2$ or a rainbow $\mathbb{S}^{(3)}_{2}$.
If $G$ contains a rainbow $S^{(3)}_2$, say $\{x_1y_1z_1, x_1y_2z_2\}$, then at least one of $\{x_1y_1z_1, x_1y_1z_2\}$ and $\{x_1z_2y_2, x_1z_2y_1\}$ is a rainbow $\mathbb{S}^{(3)}_{2}$,
so $G$ contains both a rainbow $S^{(3)}_2$ and a rainbow $\mathbb{S}^{(3)}_{2}$ in this case.
If $G$ contains a rainbow $\mathbb{S}^{(3)}_{2}$, say $\{x_1y_1z_1, x_1y_1z_2\}$, then at least one of $\{x_1y_1z_1, x_1y_2z_3\}$ and $\{x_1y_1z_2, x_1y_2z_3\}$ is a rainbow $S^{(3)}_2$,
so $G$ contains both a rainbow $S^{(3)}_2$ and a rainbow $\mathbb{S}^{(3)}_{2}$.
\end{proof}

We first prove Theorem~\ref{th:multi-messy}, which will also be used in our proof of Theorem~\ref{th:multi-tight}.

\vspace{0.3cm}
\noindent{\bf Proof of Theorem~\ref{th:multi-messy}.}~
Let $G$ be a rainbow $\mathcal{M}$-free edge-colored $K^{(3)}_{n,n,n}$ with $|C(G)|\geq 3$ and $n\geq 3$.
Let $V_1, V_2, V_3$ be the partite sets of $G$, where $V_1=\{x_1, x_2, \ldots, x_n\}$, $V_2=\{y_1, y_2, \ldots, y_n\}$ and $V_3=\{z_1, z_2, \ldots, z_n\}$.

\begin{claim}\label{cl:multi-messy-1}
$G$ contains a rainbow $\mathbb{S}^{(3)}_3$.
\end{claim}

\begin{proof}
For a contradiction, suppose that $G$ contains no rainbow $\mathbb{S}^{(3)}_3$.
By Observation~\ref{obs:S2-multipar}, $G$ contains a rainbow $\mathbb{S}^{(3)}_{2}$, say $\{x_1y_1z_1, x_1y_1z_2\}$, where $c(x_1y_1z_1)=1$ and $c(x_1y_1z_2)=2$.
Note that for any edge $e$ with $e\cap \{x_1, y_1, z_1, z_2\}=\emptyset$ we have $c(e)\in \{1,2\}$.
Indeed, if $c(e)\notin \{1,2\}$, say $e=x_2y_2z_3$, then $c(x_1y_1z_3)=c(e)$ to avoid a rainbow $\mathcal{M}$, but then there is a rainbow $\mathbb{S}^{(3)}_3$, a contradiction.

Let $f$ be an edge of color 3.
Then $1\leq |f\cap \{x_1, y_1, z_1, z_2\}|\leq 2$.
If $|f\cap \{x_1, y_1, z_1, z_2\}|=1$, then $f\cap \{x_1, y_1, z_1, z_2\}\in \big\{\{x_1\}, \{y_1\}\big\}$ to avoid a rainbow $\mathcal{M}$.
Without loss of generality, we may assume that $f=x_1y_2z_3$.
Since $c(x_2y_2z_3)\in \{1,2\}$ (note that $\{x_2, y_2, z_3\} \cap \{x_1, y_1, z_1, z_2\}=\emptyset$),
one of $\{x_2y_2z_3, y_2z_3x_1, x_1y_1z_2\}$ and $\{x_2y_2z_3, y_2z_3x_1, x_1y_1z_1\}$ is a rainbow $\mathcal{M}$, a contradiction.
If $|f\cap \{x_1, y_1, z_1, z_2\}|=2$, then $f\cap \{x_1, y_1, z_1, z_2\}\in \big\{\{x_1, z_1\}, \{x_1, z_2\}, \{y_1, z_1\}, \{y_1, z_2\}\big\}$ to avoid a rainbow $\mathbb{S}^{(3)}_3$.
By symmetry, we may assume that $f\cap \{x_1, y_1, z_1, z_2\}=\{x_1, z_1\}$ and $f=x_1y_2z_1$.
Since $c(x_2y_2z_3)\in \{1,2\}$, we have $c(x_2y_2z_3)=1$ to avoid a rainbow $\mathcal{M}$.
Then we further have $c(x_2y_3z_3)=1$ to avoid a rainbow $\mathcal{M}$.
Now no matter what color is assigned to the edge $x_2y_2z_1$,
at least one of $\{x_2y_2z_1, y_2z_1x_1, x_1y_1z_2\}$, $\{x_1y_2z_1, y_2z_1x_2, x_2y_3z_3\}$ and $\{z_2x_1y_1, x_1y_1z_1, z_1y_2x_2\}$ is a rainbow $\mathcal{M}$, a contradiction.
\end{proof}

By Claim~\ref{cl:multi-messy-1}, $G$ contains a rainbow $\mathbb{S}^{(3)}_3$, say $\{x_1y_1z_1, x_1y_1z_2, x_1y_1z_3\}$, where $c(x_1y_1z_i)=i$ for $i\in [3]$.
In order to avoid a rainbow $\mathcal{M}$, we have $c(z_ix_jy_{k})=i$ for any $i\in [3]$ and $j, k \in \{2, 3, \ldots, n\}$.
This implies that for any fixed $j_0, k_0 \in \{2, 3, \ldots, n\}$, $\{x_{j_0}y_{k_0}z_1, x_{j_0}y_{k_0}z_2, x_{j_0}y_{k_0}z_3\}$ is a rainbow $\mathbb{S}^{(3)}_3$.
Applying the same argument to different such rainbow copies of $\mathbb{S}^{(3)}_3$, we further have $c(z_ix_jy_{k})=i$ for any $i\in [3]$ and $j, k \in [n]$.
Furthermore, for any $i\in [n]\setminus \{1,2,3\}$, $x_1y_1z_i$ and two edges in $\{x_1y_1z_1, x_1y_1z_2, x_1y_1z_3\}$ form a rainbow $\mathbb{S}^{(3)}_3$.
Now for any $i\in [n]\setminus \{1,2,3\}$ and $j, k \in [n]$, we have $c(z_ix_jy_{k})=c(z_ix_1y_1)$ by an argument analogous to the one above.
This implies that we can partition $V_3$ into $|C(G)|$ parts $V_{3,1}, V_{3,2}, \ldots, V_{3,|C(G)|}$,
where $V_{3,a}\colonequals \{z\in V_3\colon\, \mbox{every edge containing $z$ has color $a$}\}$ for each $a\in C(G)$.
The proof of Theorem~\ref{th:multi-messy} is complete.
\hfill$\square$
\vspace{0.3cm}

We next prove Theorem~\ref{th:multi-tight}.

\vspace{0.3cm}
\noindent{\bf Proof of Theorem~\ref{th:multi-tight}.}~
Let $G$ be a rainbow $\mathcal{T}$-free edge-colored $K^{(3)}_{n,n,n}$ with $|C(G)|\geq 3$ and $n\geq 3$.
Let $V_1, V_2, V_3$ be the partite sets of $G$, where $V_1=\{x_1, x_2, \ldots, x_n\}$, $V_2=\{y_1, y_2, \ldots, y_n\}$ and $V_3=\{z_1, z_2, \ldots, z_n\}$.
If $G$ contains no rainbow $\mathcal{M}$, then (i) holds by Theorem~\ref{th:multi-messy}.
Thus we may assume that $G$ contains a rainbow $\mathcal{M}$.
In Claims~\ref{cl:multi-tight-1}, \ref{cl:multi-tight-2} and \ref{cl:multi-tight-3} below, we shall assume that $e_1, e_2, e_3$ form a rainbow $\mathcal{M}$ in which $e_1$ is the middle edge, so $e_2\cap e_3=\emptyset$.

\begin{claim}\label{cl:multi-tight-1}
If $e_1, e_2, e_3$ form a rainbow $\mathcal{M}$ in $G$ with $e_2\cap e_3=\emptyset$, then for every edge $e\in E(G)$ with $e\cap e_2\neq \emptyset$ and $e\cap e_3\neq \emptyset$, we have $c(e)=c(e_1)$.
\end{claim}

\begin{proof}
Without loss of generality, we may assume that $e_1=x_1y_1z_2$, $e_2=x_1y_1z_1$ and $e_3=x_2y_2z_2$, where $c(e_i)=i$ for each $i\in [3]$.
We first show that for the remaining five edges $e\in E(G)$ with $e\subseteq e_2\cup e_3$, we have $c(e)=c(e_1)=1$.
Note that the two edges $x_1y_2z_2$ and $x_2y_1z_2$ containing $z_2$ must be of color 1.
To see this, we only consider $x_1y_2z_2$ by symmetry.
If $c(x_1y_2z_2)\neq 1$, then at least one of $\{y_1x_1z_2, x_1z_2y_2, z_2y_2x_2\}$ and $\{z_1y_1x_1, y_1x_1z_2, x_1z_2y_2\}$ is a rainbow $\mathcal{T}$, a contradiction.
Thus $c(x_1y_2z_2)=1$ and $c(x_2y_1z_2)=1$.
Now applying a similar argument to the rainbow copy $\{x_2y_2z_2, y_2z_2x_1, x_1y_1z_1\}$ or $\{y_2x_2z_2, x_2z_2y_1, y_1x_1z_1\}$ of $\mathcal{M}$, we can also deduce that the remaining edges within $e_2\cup e_3$ containing $x_1$ or $y_1$ are of color 1.
For the one remaining edge $x_2y_2z_1$, we have $c(x_2y_2z_1)=1$, since otherwise at least one of $\{y_1z_1x_2, z_1x_2y_2, x_2y_2z_2\}$ and $\{x_1y_1z_1, y_1z_1x_2, z_1x_2y_2\}$ is a rainbow $\mathcal{T}$, a contradiction.

We next consider edges $e\in E(G)$ with $e\cap e_2\neq \emptyset$, $e\cap e_3\neq \emptyset$ and $e\setminus (e_2\cup e_3)\neq \emptyset$.
Note that all the six vertices within $e_2\cup e_3$ are symmetrical now.
Thus we may assume that $e=x_1y_2z_3$ without loss of generality.
If $c(x_1y_2z_3)\neq 1$, then at least one of $\{x_2z_2y_2, z_2y_2x_1, y_2x_1z_3\}$ and $\{y_1z_1x_1, z_1x_1y_2, x_1y_2z_3\}$ is a rainbow $\mathcal{T}$, a contradiction.
Thus $c(x_1y_2z_3)= 1$, and by symmetry, all edges $e\in E(G)$ with $e\cap e_2\neq \emptyset$, $e\cap e_3\neq \emptyset$ and $e\setminus (e_2\cup e_3)\neq \emptyset$ are of color 1.
This completes the proof of Claim~\ref{cl:multi-tight-1}.
\end{proof}

\begin{claim}\label{cl:multi-tight-2}
If $e_1, e_2, e_3$ form a rainbow $\mathcal{M}$ in $G$ with $e_2\cap e_3=\emptyset$, then for each $i\in \{2,3\}$, every edge $f\in E(G)$ with $f\cap e_i \neq \emptyset$ satisfies $c(f)\in \{c(e_1), c(e_i)\}$.
\end{claim}

\begin{proof}
Without loss of generality, we may assume that $e_2=x_1y_1z_1$, $e_3=x_2y_2z_2$ and $c(e_i)=i$ for each $i\in [3]$.
By Claim~\ref{cl:multi-tight-1}, for every edge $e\in E(G)$ with $e\cap e_2\neq \emptyset$ and $e\cap e_3\neq \emptyset$, we have $c(e)=c(e_1)=1$.
Suppose for a contradiction that for some $i\in \{2,3\}$, say $i=2$, there exists an edge $f\in E(G)$ with $f\cap e_2 \neq \emptyset$ and $c(f)\notin \{c(e_1), c(e_2)\}=\{1,2\}$.
Then $1\leq |f\cap e_2|\leq 2$ and $f\cap e_3=\emptyset$.

If $|f\cap e_2|=2$, say $f=z_3x_1y_1$, then $\{z_3x_1y_1, x_1y_1z_1, y_1z_1x_2\}$ is a rainbow $\mathcal{T}$, a contradiction.
Thus we have $c(f)\in \{1,2\}$ whenever $|f\cap e_2|=2$.
If $|f\cap e_2|=1$, say $f=y_3z_3x_1$, then since $c(z_3x_1y_1)\in \{1,2\}$ and $c(f)\notin \{1,2\}$, at least one of $\{y_3z_3x_1, z_3x_1y_1, x_1y_1z_1\}$ and $\{y_3z_3x_1, z_3x_1y_1, x_1y_1z_2\}$ is a rainbow $\mathcal{T}$, a contradiction.
Thus we have $c(f)\in \{1,2\}$ whenever $|f\cap e_2|=1$.
This completes the proof of Claim~\ref{cl:multi-tight-2}.
\end{proof}

We now prove a claim which can be intuitively interpreted as follows.
Let $\{e_1, e_2, e_3\}$ be a rainbow $\mathcal{M}$ with $e_2\cap e_3=\emptyset$.
For each $i\in \{2,3\}$, consider all edges $h$ with $c(h)\notin \{c(e_1), c(e_i)\}$.
We can show that each such $h$ together with $e_i$ and some edge of color $c(e_1)$ forms a rainbow $\mathcal{M}$ with $h\cap e_i=\emptyset$.
Then the conclusion in Claim~\ref{cl:multi-tight-1} also holds for these rainbow copies of $\mathcal{M}$.

\begin{claim}\label{cl:multi-tight-3}
If $e_1, e_2, e_3$ form a rainbow $\mathcal{M}$ in $G$ with $e_2\cap e_3=\emptyset$, then for each $i\in \{2,3\}$ and each edge $h\in E(G)$ with $c(h)\notin \{c(e_1), c(e_i)\}$, the following statements hold:
\begin{itemize}
\item[{\rm (a)}] $h\cap e_i=\emptyset$;

\item[{\rm (b)}] for all edges $e$ with $e\cap e_i\neq \emptyset$ and $e\cap h\neq \emptyset$, we have $c(e)=c(e_1)$.
\end{itemize}
\end{claim}

\begin{proof}
Without loss of generality, we may assume that $e_2=x_1y_1z_1$, $e_3=x_2y_2z_2$ and $c(e_i)=i$ for each $i\in [3]$.
Let $h\in E(G)$ be an edge with $c(h)\notin \{c(e_1), c(e_i)\}=\{1, i\}$ for some $i\in \{2,3\}$, say $i=2$.
By Claim~\ref{cl:multi-tight-2}, we have $h\cap e_2 = \emptyset$, so (a) holds.
If $h=e_3$, then (b) follows from Claim~\ref{cl:multi-tight-1}, so we may assume that $h\neq e_3$.
In particular, $\{x_2, y_2, z_2\}\setminus h\neq \emptyset$.
Without loss of generality, we may assume that $x_2\notin h$ and $h=x_{j_1}y_{j_2}z_{j_3}$, where $j_1\notin \{1,2\}$, $j_2\neq 1$ and $j_3\neq 1$.
By Claims~\ref{cl:multi-tight-1} and \ref{cl:multi-tight-2}, we have $c(y_{j_2}z_1x_2)=1$ and $c(x_{j_1}y_{j_2}z_1)\in \{1,2\}$.
If $c(x_{j_1}y_{j_2}z_1)=2$, then $\{z_{j_3}x_{j_1}y_{j_2}, x_{j_1}y_{j_2}z_1, y_{j_2}z_1x_2\}$ is a rainbow $\mathcal{T}$, a contradiction.
Thus $c(x_{j_1}y_{j_2}z_1)=1=c(e_1)$, so $\{z_{j_3}x_{j_1}y_{j_2}, x_{j_1}y_{j_2}z_1, z_1y_1x_1\}$ (i.e., $\{h, x_{j_1}y_{j_2}z_1, e_2\}$) is a rainbow $\mathcal{M}$.
Applying Claim~\ref{cl:multi-tight-1} to this rainbow $\mathcal{M}$, we can deduce that (b) holds.
\end{proof}

Recall that $G$ contains a rainbow $\mathcal{M}$.
Without loss of generality, we may assume that $e_1=x_1y_1z_2$, $e_2=x_1y_1z_1$ and $e_3=x_2y_2z_2$ form a rainbow $\mathcal{M}$, where $c(e_i)=i$ for each $i\in [3]$.
We next prove a claim which can be intuitively interpreted as follows.
For each $2\leq i<j\leq |C(G)|$, consider all $($unordered$)$ pairs $(e'_i,e'_j)$ of edges with $c(e'_i)=i$ and $c(e'_j)=j$.
We can show that $e'_i\cap e'_j=\emptyset$, and for each such pair $(e'_i,e'_j)$, there exists an edge of color $c(e_1)$ which together with $e'_i, e'_j$ forms a rainbow $\mathcal{M}$.
Then the conclusion in Claim~\ref{cl:multi-tight-1} also holds for these rainbow copies of $\mathcal{M}$.

\begin{claim}\label{cl:multi-tight-4}
For all $2\leq i<j\leq |C(G)|$ and all pairs $(e'_i,e'_j)$ of edges with $c(e'_i)=i$ and $c(e'_j)=j$, the following statements hold:
\begin{itemize}
\item[{\rm (a)}] $e'_i\cap e'_j=\emptyset$;

\item[{\rm (b)}] for all edges $e$ with $e\cap e'_i\neq \emptyset$ and $e\cap e'_j\neq \emptyset$, we have $c(e)=c(e_1)$.
\end{itemize}
\end{claim}

\begin{proof}
We will prove the result by induction on $j\geq 3$.
Note that by Claim~\ref{cl:multi-tight-3}, for all edges $e'_2\in E(G)$ with $c(e'_2)=2$, we have
\begin{itemize}
\item[{\rm (a$'$)}] $e'_2\cap e_3=\emptyset$;

\item[{\rm (b$'$)}] for all edges $e$ with $e\cap e_3\neq \emptyset$ and $e\cap e'_2\neq \emptyset$, we have $c(e)=c(e_1)$.
\end{itemize}
In particular, for any edge $e'_2$ with $c(e'_2)=2$, there exists an edge of color $c(e_1)$ which together with $e'_2$ and $e_3$ forms a rainbow $\mathcal{M}$ with $e'_2\cap e_3=\emptyset$.
Indeed, since $e'_2\cap e_3=\emptyset$, we can choose an edge meeting one of them in two vertices and the other in one; (b$'$) shows that this edge has color $c(e_1)$.
Now applying Claim~\ref{cl:multi-tight-3} to such a rainbow $\mathcal{M}$, we can deduce that for every pair $(e'_2,e'_3)$ of edges with $c(e'_2)=2$ and $c(e'_3)=3$, we have
\begin{itemize}
\item[{\rm (a$''$)}] $e'_3\cap e'_2=\emptyset$;

\item[{\rm (b$''$)}] for all edges $e$ with $e\cap e'_2\neq \emptyset$ and $e\cap e'_3\neq \emptyset$, we have $c(e)=c(e_1)$.
\end{itemize}
This implies that Claim~\ref{cl:multi-tight-4} holds for $j=3$.
If $|C(G)|=3$, then we are done, so we may assume that $|C(G)|\geq 4$.

Assume that the result holds for $3\leq j<|C(G)|$, and we shall prove it for $j+1$.
For any $2\leq i<j+1$, let $k\in \{2, \ldots, j\}\setminus \{i\}$ (here $i<k$ and $k<i$ are both possible).
By the induction hypothesis, for any pair $(e'_{i}, e'_{k})$ of edges with $c(e'_i)=i$ and $c(e'_{k})=k$, the conclusions (a) and (b) hold.
In particular, we have $e'_i\cap e'_{k}=\emptyset$, 
and as in the base case, there exists an edge of color $c(e_1)$ which together with $e'_i$ and $e'_{k}$ forms a rainbow $\mathcal{M}$.
Applying Claim~\ref{cl:multi-tight-3} to such a rainbow $\mathcal{M}$, we can deduce that for every pair $(e'_i,e'_{j+1})$ of edges with $c(e'_i)=i$ and $c(e'_{j+1})=j+1\notin \{1, i\}$, we have
\begin{itemize}
\item[{\rm (a$'''$)}] $e'_{j+1}\cap e'_i=\emptyset$;

\item[{\rm (b$'''$)}] for all edges $e$ with $e\cap e'_i\neq \emptyset$ and $e\cap e'_{j+1}\neq \emptyset$, we have $c(e)=c(e_1)$.
\end{itemize}
This completes the proof.
\end{proof}

By Claim~\ref{cl:multi-tight-4}~(a), for all $2\leq i<j\leq |C(G)|$ and all pairs $(e'_i,e'_j)$ of edges with $c(e'_i)=i$ and $c(e'_j)=j$, we have $e'_i\cap e'_j=\emptyset$.
For each $\ell\in [3]$ and $i\in \{2, 3, \ldots, |C(G)|\}$, define $U_{\ell,i}\colonequals \{v\in V_{\ell}\colon\, \mbox{$v$ lies in an edge of color $i$}\}$ and $U_{\ell,1}\colonequals V_{\ell}\setminus \big(\bigcup_{2\leq i\leq |C(G)|}U_{\ell, i}\big)$.
Then these sets are pairwise disjoint.
For each $\ell\in [3]$, let $V_{\ell,2}= U_{\ell,2}\cup U_{\ell,1}$, and let $V_{\ell,i}= U_{\ell,i}$ for $i\in \{3, \ldots, |C(G)|\}$.
Then for each $i\in \{2, 3, \ldots, |C(G)|\}$, all edges within $V_{1,i}\cup V_{2,i}\cup V_{3,i}$ are of color 1 or $i$.
By Claim~\ref{cl:multi-tight-4}~(b), all the remaining edges are of color 1.
This implies that Theorem~\ref{th:multi-tight}~(ii) holds.
The proof of Theorem~\ref{th:multi-tight} is complete.
\hfill$\square$
\vspace{0.3cm}

We now prove Theorem~\ref{th:multi-loose}.

\vspace{0.3cm}
\noindent{\bf Proof of Theorem~\ref{th:multi-loose}.}~
Let $G$ be a rainbow $\mathcal{L}$-free edge-colored $K^{(3)}_{n,n,n}$ with $|C(G)|\geq 3$ and $n\geq 3$.
Let $V_1, V_2, V_3$ be the partite sets of $G$, where $V_1=\{x_1, x_2, \ldots, x_n\}$, $V_2=\{y_1, y_2, \ldots, y_n\}$ and $V_3=\{z_1, z_2, \ldots, z_n\}$.

\begin{claim}\label{cl:multi-loose-1}
$G$ contains neither a rainbow $S^{(3)}_3$ nor a rainbow $S^{(3)}_2\cup S^{(3)}_1$.
\end{claim}

\begin{proof}
We first show that $G$ contains no rainbow $S^{(3)}_3$.
For a contradiction, suppose that $G$ contains a rainbow $S^{(3)}_3$, say $\{x_1y_1z_1, x_1y_2z_2, x_1y_3z_3\}$.
In order to avoid a rainbow $\mathcal{L}$, we have $c(x_2y_2z_1)=c(x_1y_3z_3)$ and $c(x_3y_3z_2)=c(x_1y_1z_1)$.
But then $\{x_3y_3z_2, z_2x_1y_2, y_2z_1x_2\}$ is a rainbow $\mathcal{L}$, a contradiction.

We next show that $G$ contains no rainbow $S^{(3)}_2\cup S^{(3)}_1$.
Suppose for a contradiction that $G$ contains a rainbow $S^{(3)}_2\cup S^{(3)}_1$, say $\{x_1y_1z_1, x_1y_2z_2, x_2y_3z_3\}$.
To avoid a rainbow $\mathcal{L}$, we have $c(y_2x_3z_3)=c(x_1y_2z_2)$.
But then no matter what color is assigned to $x_2y_2z_1$, at least one of $\{x_1z_2y_2, y_2z_1x_2, x_2y_3z_3\}$, $\{x_1y_1z_1, z_1y_2x_2, x_2y_3z_3\}$ and $\{x_1y_1z_1, z_1x_2y_2, y_2x_3z_3\}$ is a rainbow $\mathcal{L}$, a contradiction.
\end{proof}

Recall that $\Delta^{c}(G)\colonequals \max_{v\in V(G)}d^{c}(v)$ is the maximum color degree of $G$.

\begin{claim}\label{cl:multi-loose-2}
$\Delta^{c}(G)\geq 3$.
\end{claim}

\begin{proof}
For a contradiction, suppose that $\Delta^{c}(G)\leq 2$.
By Observation~\ref{obs:S2-multipar}, $G$ contains a rainbow copy $\{e_1, e_2\}$ of $S^{(3)}_{2}$, say $e_1=x_1y_1z_1$, $e_2=x_1y_2z_2$, $c(e_1)=1$ and $c(e_2)=2$.
Let $e_3\in E(G)$ be an edge of color 3.
By Claim~\ref{cl:multi-loose-1}, $G$ contains no rainbow $S^{(3)}_2\cup S^{(3)}_1$, so $e_3\cap (e_1\cup e_2)\neq \emptyset$.
Combining with $\Delta^{c}(G)\leq 2$ and the rainbow $\mathcal{L}$-freeness of $G$, we have $x_1\notin e_3$ and $|e_3\cap (e_1\cup e_2)|=2$,
so $e_3\cap (e_1\cup e_2)\in \big\{\{y_1, z_1\}, \{y_2, z_2\}, \{y_1, z_2\}, \{y_2, z_1\}\big\}$.

If $e_3\cap (e_1\cup e_2)=\{y_1, z_1\}$, say $e_3=x_2y_1z_1$, then $c(x_3y_3z_1)\in \{1,3\}$ since $d^{c}(z_1)\leq \Delta^{c}(G)\leq 2$, and $c(x_3y_3z_1)\in \{1,2\}$ since $G$ is rainbow $\mathcal{L}$-free.
Hence, $c(x_3y_3z_1)=1$, but then $\{z_1y_1x_2, z_1y_3x_3, x_1y_2z_2\}$ is a rainbow $S^{(3)}_2\cup S^{(3)}_1$, contradicting Claim~\ref{cl:multi-loose-1}.
Thus $e_3\cap (e_1\cup e_2)\neq \{y_1, z_1\}$, and by symmetry, we also have $e_3\cap (e_1\cup e_2)\neq \{y_2, z_2\}$.

Suppose now $e_3\cap (e_1\cup e_2)=\{y_1, z_2\}$, say $e_3=x_2y_1z_2$.
Since $d^{c}(x_1)\leq \Delta^{c}(G)\leq 2$, we have $c(x_1y_3z_3)\in \{1,2\}$.
But then one of $\{z_3y_3x_1, x_1y_2z_2, z_2y_1x_2\}$ and $\{z_3y_3x_1, x_1z_1y_1, y_1x_2z_2\}$ is a rainbow $\mathcal{L}$, a contradiction.
Thus $e_3\cap (e_1\cup e_2)\neq \{y_1, z_2\}$, and $e_3\cap (e_1\cup e_2)\neq \{y_2, z_1\}$ by symmetry.
This contradiction completes the proof of Claim~\ref{cl:multi-loose-2}.
\end{proof}

\begin{claim}\label{cl:multi-loose-3}
There exist three edges $e_1, e_2, e_3$ of distinct colors with $|e_1\cap e_2|=2$ and $|e_1\cap e_3|=|e_2\cap e_3|=|e_1\cap e_2\cap e_3|=1$.
\end{claim}

\begin{proof}
By Claim~\ref{cl:multi-loose-2}, there exists a vertex, say $x_1$, with color degree at least 3.
Without loss of generality, we may assume that $x_1$ is incident with three edges $e_1, e_2, e_3$ with $c(e_i)=i$ for $i\in [3]$.
By Claim~\ref{cl:multi-loose-1}, $G$ contains no rainbow $S^{(3)}_3$.
Hence, two of the three edges $e_1, e_2, e_3$, say $e_1$ and $e_2$, satisfy $|e_1\cap e_2|=2$.
Without loss of generality, we may assume that $e_1=x_1y_1z_1$ and $e_2=x_1y_1z_2$.
If $e_1\cap e_3=e_2\cap e_3=\{x_1\}$, then $e_1\cap e_2\cap e_3=\{x_1\}$.
This implies that $|e_1\cap e_3|=|e_2\cap e_3|=|e_1\cap e_2\cap e_3|=1$, and we are done.
Thus we may assume that $e_1\cap e_3\neq \{x_1\}$ or $e_2\cap e_3\neq \{x_1\}$, say $e_1\cap e_3\neq \{x_1\}$.
Then $e_1\cap e_3=\{x_1, z_1\}$ or $e_1\cap e_3=\{x_1, y_1\}$.

Next, we show that if $c(x_1y_iz_1)\notin \{1,2\}$ or $c(x_1y_iz_2)\notin \{1,2\}$ for some $i\in [n]\setminus \{1\}$, then the conclusion of Claim~\ref{cl:multi-loose-3} holds.
By symmetry, it suffices to consider $c(x_1y_iz_1)$.
Assume that $c(x_1y_iz_1)\notin \{1,2\}$ for some $i\in [n]\setminus \{1\}$, say $c(x_1y_2z_1)=a$ for some $a\notin \{1,2\}$.
Then $c(x_1y_3z_3)\in \{2,a\}$ since otherwise $\{x_1y_1z_2, x_1y_2z_1, x_1y_3z_3\}$ is a rainbow $S^{(3)}_3$, contradicting Claim~\ref{cl:multi-loose-1}.
If $c(x_1y_3z_3)=2$ (resp., $c(x_1y_3z_3)=a$), then $x_1y_1z_1, x_1y_2z_1, x_1y_3z_3$ (resp., $x_1y_1z_1, x_1y_1z_2, x_1y_3z_3$) are three edges satisfying the conclusion of the claim, and we are done.

Now we may assume that $c(x_1y_iz_1)\in \{1,2\}$ and $c(x_1y_iz_2)\in \{1,2\}$ for all $i\in [n]\setminus \{1\}$.
Recall that $e_1\cap e_3=\{x_1, z_1\}$ or $e_1\cap e_3=\{x_1, y_1\}$.
Since $c(e_3)=3$ and $c(x_1y_iz_1)\in \{1,2\}$ for all $i\in [n]\setminus \{1\}$, we have $e_1\cap e_3=\{x_1, y_1\}$, say $e_3=x_1y_1z_3$.
Then since $c(x_1y_2z_1)\in \{1,2\}$ and $c(x_1y_3z_2)\in \{1,2\}$, we have $c(x_1y_2z_1)= c(x_1y_3z_2)$; otherwise $\{x_1y_2z_1, x_1y_3z_2, x_1y_1z_3\}$ is a rainbow $S^{(3)}_3$, contradicting Claim~\ref{cl:multi-loose-1}.
Without loss of generality, we may assume that $c(x_1y_2z_1)= c(x_1y_3z_2)=1$.
Now $x_1y_1z_2, x_1y_1z_3, x_1y_2z_1$ are three edges satisfying the conclusion of the claim.
This completes the proof of Claim~\ref{cl:multi-loose-3}.
\end{proof}

\begin{claim}\label{cl:multi-loose-4}
There exist two edges $e_1, e_2$, three distinct colors $c_1, c_2, c_3$ and a vertex $v$ with $v\in e_1\cap e_2$, $|e_1\cap e_2|=2$, $c(e_1)=c_1$ and $c(e_2)=c_2$ such that, for all edges $e\in E(G)$ with $e\cap (e_1\cup e_2)=\{v\}$, we have $c(e)=c_3$.
\end{claim}

\begin{proof}
By Claim~\ref{cl:multi-loose-3}, there exist three edges $e_1, e_2, e_3$ of distinct colors with $|e_1\cap e_2|=2$ and $|e_1\cap e_3|=|e_2\cap e_3|=|e_1\cap e_2\cap e_3|=1$.
Then $e_1\cap e_3=e_2\cap e_3=e_1\cap e_2\cap e_3$.
Without loss of generality, we may assume that $e_1=x_1y_1z_1$, $e_2=x_1y_1z_2$, $e_3=x_1y_2z_3$ and $c(e_i)=i$ for $i\in [3]$.
We shall show that every edge $e$ with $e\cap (e_1\cup e_2)=\{x_1\}$ satisfies $c(e)=3$, which implies that Claim~\ref{cl:multi-loose-4} holds with $v=x_1$ and $c_3=3$.
By Claim~\ref{cl:multi-loose-1}, $G$ contains no rainbow $S^{(3)}_3$.
Thus we have $c(e)=3$ for every edge $e$ with $e\cap (e_1\cup e_2\cup e_3)=\{x_1\}$.
It remains to consider edges $e$ of the form $x_1y_2z_i$ ($i\in [n]\setminus [3]$) or $x_1y_jz_3$ ($j\in [n]\setminus [2]$).

Note that $c(x_2y_jz_3)=3$ for any $j\in [n]\setminus [2]$, since otherwise at least one of $\{x_2y_jz_3, z_3y_2x_1,$ $x_1y_1z_1\}$ and $\{x_2y_jz_3, z_3y_2x_1, x_1y_1z_2\}$ is a rainbow $\mathcal{L}$.
Then $c(x_1y_2z_i)=3$ for every $i\in [n]\setminus [3]$, since otherwise at least one of $\{x_1y_1z_1, x_1y_2z_i, x_2y_3z_3\}$ and $\{x_1y_1z_2, x_1y_2z_i, x_2y_3z_3\}$ is a rainbow $S^{(3)}_2\cup S^{(3)}_1$, contradicting Claim~\ref{cl:multi-loose-1}.
Moreover, we have $c(x_2y_2z_1)=3$, since otherwise at least one of $\{z_1x_2y_2, y_2z_3x_1, x_1y_1z_2\}$ and $\{x_1y_1z_1, z_1y_2x_2, x_2y_3z_3\}$ is a rainbow $\mathcal{L}$.
By symmetry, we also have $c(x_2y_2z_2)=3$.
Then $c(x_1y_jz_3)=3$ for every $j\in [n]\setminus [2]$, since otherwise at least one of $\{x_1y_1z_2, x_1y_jz_3, x_2y_2z_1\}$ and $\{x_1y_1z_1, x_1y_jz_3, x_2y_2z_2\}$ is a rainbow $S^{(3)}_2\cup S^{(3)}_1$, contradicting Claim~\ref{cl:multi-loose-1}.
The result follows.
\end{proof}

\begin{claim}\label{cl:multi-loose-5}
There exist two edges $e_1, e_2$ and three distinct colors $c_1, c_2, c_3$ with $|e_1\cap e_2|=2$, $c(e_1)=c_1$ and $c(e_2)=c_2$ such that, for all edges $e\in E(G)$ with $|e\cap (e_1\cup e_2)|\leq 1$, we have $c(e)=c_3$.
\end{claim}

\begin{proof}
By Claim~\ref{cl:multi-loose-4}, we may assume that $e_1=x_1y_1z_1$ and $e_2=x_1y_1z_2$ are two edges with $c(e_1)=1$ and $c(e_2)=2$ such that, for all edges $e\in E(G)$ with $e\cap (e_1\cup e_2)=\{x_1\}$, we have $c(e)=3$.
If $e'\cap (e_1\cup e_2)=\emptyset$, say $e'=x_2y_2z_3$, then combining with $c(x_1y_3z_3)=3$ (note that $\{x_1, y_3, z_3\}\cap (e_1\cup e_2)=\{x_1\}$), we have $c(e')=3$ since otherwise at least one of $\{z_1y_1x_1, x_1y_3z_3, z_3y_2x_2\}$ and $\{z_2y_1x_1, x_1y_3z_3, z_3y_2x_2\}$ is a rainbow $\mathcal{L}$.
If $e''\cap (e_1\cup e_2)=\{y_1\}$, say $e''=x_2y_1z_3$, then combining with $c(x_3y_2z_3)=3$ (note that $\{x_3, y_2, z_3\}\cap (e_1\cup e_2)=\emptyset$), we have $c(e'')=3$ since otherwise at least one of $\{x_1z_1y_1, y_1x_2z_3, z_3y_2x_3\}$ and $\{x_1z_2y_1, y_1x_2z_3, z_3y_2x_3\}$ is a rainbow $\mathcal{L}$.
If $e'''\cap (e_1\cup e_2)\in \{\{z_1\}, \{z_2\}\}$, say $e'''=x_2y_2z_1$, then combining with $c(x_1y_2z_3)=3$ (note that $\{x_1, y_2, z_3\}\cap (e_1\cup e_2)=\{x_1\}$) and $c(x_2y_3z_3)=3$ (note that $\{x_2, y_3, z_3\}\cap (e_1\cup e_2)=\emptyset$), we have $c(e''')=3$ since otherwise at least one of $\{z_2y_1x_1, x_1z_3y_2, y_2x_2z_1\}$ and $\{x_1y_1z_1, z_1y_2x_2, x_2y_3z_3\}$ is a rainbow $\mathcal{L}$.
Therefore, for all edges $e\in E(G)$ with $|e\cap (e_1\cup e_2)|\leq 1$, we have $c(e)=3$.
\end{proof}

By Claim~\ref{cl:multi-loose-5}, we may assume that $e_1=x_1y_1z_1$ and $e_2=x_1y_1z_2$ are two edges with $c(e_1)=1$ and $c(e_2)=2$ such that, for all edges $e\in E(G)$ with $|e\cap (e_1\cup e_2)|\leq 1$, we have $c(e)=3$.

\begin{claim}\label{cl:multi-loose-6}
For any edge $e\in E(G)$ with $c(e)\notin \{1,3\}$ $($resp., $c(e)\notin \{2,3\}$$)$, we have $|e\cap e_1|=2$ $($resp., $|e\cap e_2|=2$$)$.
In particular, if $|C(G)|\geq 4$, then for every edge $e\in E(G)$ with $c(e)\notin \{1,2,3\}$, we have $\{x_1, y_1\}\subseteq e$.
\end{claim}

\begin{proof}
By symmetry, we only show that if $c(e)\notin \{1,3\}$, then $|e\cap e_1|=2$.
If $e\cap e_1=\emptyset$, then $|e\cap (e_1\cup e_2)|\leq 1$, so $c(e)=3\in \{1,3\}$.
If $|e\cap e_1|=1$, then $e\setminus (e_1\cup e_2)\neq \emptyset$.
We choose $u\in e\setminus (e_1\cup e_2)$ and $v,w\in V(G)\setminus (e_1\cup e_2\cup e)$ with $uvw\in E(G)$, which is possible since $n\geq 3$ and $G$ is 3-partite.
Note that $\{u,v,w\}\cap (e_1\cup e_2)=\emptyset$, so $c(uvw)=3$.
Then $c(e)\in \{1,3\}$ since otherwise $\{e_1, e, uvw\}$ is a rainbow $\mathcal{L}$.
Therefore, if $c(e)\notin \{1,3\}$, then $|e\cap e_1|=2$.
By symmetry, if $c(e)\notin \{2,3\}$, then $|e\cap e_2|=2$.
Moreover, if $c(e)\notin \{1,2,3\}$, then $|e\cap e_1|=|e\cap e_2|=2$.
Since $z_1$ and $z_2$ belong to the same partite set, we have $|e\cap \{z_1, z_2\}|\leq 1$.
Combining with $|e\cap e_1|=|e\cap e_2|=2$, we must have $\{x_1, y_1\}\subseteq e$.
\end{proof}

In the following, we assume that Theorem~\ref{th:multi-loose}~(i) does not hold, and we shall show that (ii) or (iii) holds.
If Theorem~\ref{th:multi-loose}~(i) does not hold, then $G$ contains an edge $h$ with $c(h)\neq 3$ and $|h\cap \{x_1,y_1\}|\leq 1$ (here we choose color 3 to be the color $i$ in Theorem~\ref{th:multi-loose}~(i)).
By Claim~\ref{cl:multi-loose-6}, we have either $c(h)=1$ and $|h\cap e_2|=2$, or $c(h)=2$ and $|h\cap e_1|=2$.
Without loss of generality, we may assume that $c(h)=2$ and $|h\cap e_1|=2$, say $h=x_1y_2z_1$.
By an argument analogous to that in Claim~\ref{cl:multi-loose-6}, we can deduce that if $e$ is an edge with $c(e)\notin \{2,3\}$, then $|e\cap e_2|=|e\cap h|=2$.
Hence, except for $e_1$, only $x_1y_2z_2$ is possible to have a color in $C(G)\setminus \{2,3\}$.
Since $|\{x_1, y_2, z_2\}\cap e_1|=1$, we have $c(x_1y_2z_2)\in \{1,3\}$ by Claim~\ref{cl:multi-loose-6}.
In particular, this implies that $|C(G)|=3$.

If $c(x_1y_2z_2)=3$, then $e_1$ is the unique edge of color 1 from the above argument, and by Claim~\ref{cl:multi-loose-6}, every edge $f$ of color 2 satisfies $|e_1\cap f|=2$.
This implies that Theorem~\ref{th:multi-loose}~(ii) holds.
If $c(x_1y_2z_2)=1$, then by an argument analogous to that in Claim~\ref{cl:multi-loose-6}, we can deduce that if $e$ is an edge with $c(e)\notin \{1,3\}$, then $|e\cap e_1|=|e\cap \{x_1, y_2, z_2\}|=2$.
Thus $G$ contains no edge of color 2 except for $e_2$ and $h$.
This implies that Theorem~\ref{th:multi-loose}~(iii) holds.
\hfill$\square$
\vspace{0.3cm}

In the following, we prove the Ramsey-type results presented in Section~\ref{subsec:multipartite}.

\vspace{0.3cm}
\noindent{\bf Proof of Proposition~\ref{prop:constrained-multipar}.}~
If $H$ and $G$ satisfy the condition in the proposition, then by Theorem~\ref{th:canonical-multipar}, there exists an integer $n$ such that every edge-coloring of $K^{(r)}_{n, \ldots, n}$ contains a $J$-canonical edge-colored $K^{(r)}_{t, \ldots, t}$ for some $J\subseteq [r]$, and thus contains either a monochromatic copy of $H$ or a rainbow copy of $G$.
This implies that $f'(H,G)$ exists.
Now suppose that for all positive integers $t$, there exists some subset $J_t\subseteq [r]$ such that
$H$ does not occur as a monochromatic subgraph of any $J_t$-canonical edge-coloring of $K^{(r)}_{t, \ldots, t}$,
and $G$  does not occur as a rainbow subgraph of any $J_t$-canonical edge-coloring of $K^{(r)}_{t, \ldots, t}$.
Then for every positive integer $n$, we can take a $J_n$-canonical edge-coloring of $K^{(r)}_{n, \ldots, n}$ that contains neither a monochromatic copy of $H$ nor a rainbow copy of $G$.
This implies that $f'(H,G)$ does not exist.
\hfill$\square$
\vspace{0.3cm}

\vspace{0.3cm}
\noindent{\bf Proof of Theorem~\ref{th:Ram-tight-messy-multipar}.}~
Since $|E(\mathcal{T})|=|E(\mathcal{M})|=3$, there is no rainbow $\mathcal{T}$ or $\mathcal{M}$ in a 2-edge-colored $K^{(3)}_{n,n,n}$.
Hence, we have $f'(H, \mathcal{T})\geq R'_2(H)$ and $f'(H, \mathcal{M})\geq R'_2(H)$.
Next, we show that $f'(H, \mathcal{T})\leq R'_2(H)$ and $f'(H, \mathcal{M})\leq R'_2(H)$.
If $E(H)=\emptyset$, then the result holds trivially, so we may assume that $E(H)\neq \emptyset$.
It suffices to consider 3-partite 3-graphs $H$ without isolated vertices.
Indeed, let $H^{\ast}$ be the subgraph obtained from $H$ by deleting its isolated vertices.
Then $R'_2(H^{\ast})\leq R'_2(H)$.
Hence, if we can show that $f'(H^{\ast}, \mathcal{T})\leq R'_2(H^{\ast})$, then we also have $f'(H^{\ast}, \mathcal{T})\leq R'_2(H)$.
This implies that every edge-colored $K^{(3)}_{n,n,n}$ with $n=R'_2(H)$ contains either a monochromatic $H^{\ast}$ or a rainbow $\mathcal{T}$.
Moreover, $n=R'_2(H)$ implies that $|V(K^{(3)}_{n,n,n})|\geq |V(H)|$.
Hence, if $K^{(3)}_{n,n,n}$ contains a monochromatic $H^{\ast}$, then it also contains a monochromatic $H$.
This implies that $f'(H, \mathcal{T})\leq R'_2(H)$.
The same argument also applies to $\mathcal{M}$.
Therefore, we may assume that $H$ contains no isolated vertices.

Let $H$ be a monochromatic subgraph in a $J$-canonical edge-coloring of $K^{(3)}_{t, t, t}$ with $|J|=1$ for some positive integer $t$.
Then $H$ is a subgraph of $K^{(3)}_{1,t,t}$ and thus $H$ is connected (recall that $H$ contains no isolated vertices).
We may choose $t$ to be the smallest integer such that $H\subseteq K^{(3)}_{1,t,t}$.
Then $R'_2(H)\geq t$.

Let $G$ be an edge-colored $K^{(3)}_{n,n,n}$ with partite sets $V_1$, $V_2$ and $V_3$, where $n=R'_2(H)$.
Suppose that $G$ contains no rainbow $\mathcal{T}$, and we shall show that $G$ contains a monochromatic $H$.
If $|C(G)|\leq 2$, then $G$ contains a monochromatic $H$ since $n= R'_2(H)$.
Thus we may assume that $|C(G)|\geq 3$.
If Theorem~\ref{th:multi-tight}~(i) holds, then $G$ contains a monochromatic $K^{(3)}_{1,n,n}$, and thus contains a monochromatic $H$ since $n\geq t$.
If Theorem~\ref{th:multi-tight}~(ii) holds, then for every $\ell\in [3]$, we can partition $V_{\ell}$ into $|C(G)|-1$ parts $V_{\ell,2}, V_{\ell,3}, \ldots, V_{\ell,|C(G)|}$ such that all edges within $V_{1,i}\cup V_{2,i}\cup V_{3,i}$ are of color 1 or $i$ for every $i\in \{2, 3, \ldots, |C(G)|\}$, and all the remaining edges are of color 1.
Let $G'$ be an auxiliary 2-edge-colored $K^{(3)}_{n,n,n}$ obtained from $G$ by recoloring all edges of colors in $\{3, \ldots, |C(G)|\}$ with color 2.
Then $G'$ contains a monochromatic $H$ since $n=R'_2(H)$.
If this monochromatic $H$ has color 1, then $G$ also contains a monochromatic $H$ of color 1.
If this monochromatic $H$ has color 2, then since $H$ is connected, there exists an $i_0\in \{2, \ldots, |C(G)|\}$ such that $H$ is contained in $G[V_{1,i_0}\cup V_{2,i_0}\cup V_{3,i_0}]$.
This implies that $G$ contains a monochromatic $H$ of color $i_0$.
In either case, $G$ contains a monochromatic $H$, which implies that $f'(H, \mathcal{T})\leq R'_2(H)$.

For $f'(H, \mathcal{M})$, note that by Theorem~\ref{th:multi-messy}, a rainbow $\mathcal{M}$-free edge-colored $K^{(3)}_{n,n,n}$ with at least three colors also has the structure as described in Theorem~\ref{th:multi-tight}~(i).
Thus we can also deduce that $f'(H, \mathcal{M})\leq R'_2(H)$ by the same argument as above.
\hfill$\square$
\vspace{0.3cm}

\noindent{\bf Proof of Theorem~\ref{th:Ram-loose-multipar}.}~
Since $|E(\mathcal{L})|=3$, there is no rainbow $\mathcal{L}$ in any 2-edge-colored $K^{(3)}_{n,n,n}$.
Thus we have $f'(H, \mathcal{L})\geq R'_2(H)$.
We next show that $f'(H, \mathcal{L})\leq R'_2(H)$, and thus $f'(H, \mathcal{L})= R'_2(H)$.

Let $G$ be an edge-colored $K^{(3)}_{n,n,n}$ with partite sets $V_1$, $V_2$ and $V_3$, where $n=R'_2(H)\geq t(H)+1$.
Suppose that $G$ contains no rainbow $\mathcal{L}$, and we shall show that $G$ contains a monochromatic $H$.
If $|C(G)|\leq 2$, then $G$ contains a monochromatic $H$ since $n= R'_2(H)$.
Thus we may assume that $|C(G)|\geq 3$.
Note that in each of the three structures described in Theorem~\ref{th:multi-loose}~(i), (ii) and (iii), there exists a monochromatic $K^{(3)}_{n-1,n-1,n-1}$.
Indeed, in case (i), delete the two vertices $x_1, y_1$ and one arbitrary vertex from the third part;
in case (ii), delete the three vertices of the edge $e$;
in case (iii), delete $x_1$ and one arbitrary vertex from each of the other two parts.
The resulting $K^{(3)}_{n-1,n-1,n-1}$ is monochromatic.
Since $n-1=R'_2(H)-1\geq t(H)$, there is a monochromatic $H$ in $G$.
This completes the proof of Theorem~\ref{th:Ram-loose-multipar}.
\hfill$\square$
\vspace{0.3cm}

\noindent{\bf Proof of Theorem~\ref{th:anti-Ram-multipar}.}~
For the lower bounds, we construct three edge-colored copies $G_1$, $G_2$ and $G_3$ of $K^{(3)}_{n,n,n}$ as follows.
Let $V_1, V_2, V_3$ be the partite sets with $V_1=\{x_1, x_2, \ldots, x_n\}$, $V_2=\{y_1, y_2, \ldots, y_n\}$ and $V_3=\{z_1, z_2, \ldots, z_n\}$.
In $G_1$, we color the edges such that $c(x_iy_iz_i)=i$ for each $i\in [n]$, and all the remaining edges are of color $n+1$.
In $G_2$, we color the edges such that for each $i\in [n]$, all edges containing $x_i$ are of color $i$.
In $G_3$, we color the edges such that $c(x_1y_1z_i)=i$ for each $i\in [n]$, and all the remaining edges are of color $n+1$.
Since $|V(\mathcal{T})|=5$, every copy of $\mathcal{T}$ in $G_1$ contains at most one edge of a color in $[n]$.
Thus it has at least two edges of color $n+1$, so there is no rainbow $\mathcal{T}$ in $G_1$.
Hence, $ar(K^{(3)}_{n, n, n},\mathcal{T})\geq |C(G_1)|+1=n+2$.
Moreover, every vertex in the middle edge of $\mathcal{M}$ has degree 2 in $\mathcal{M}$.
Since for every vertex $x\in V_1$, all edges containing $x$ in $G_2$ are of the same color, there is no rainbow $\mathcal{M}$ in $G_2$.
Thus $ar(K^{(3)}_{n, n, n},\mathcal{M})\geq |C(G_2)|+1=n+1$.
Furthermore, since every copy of $\mathcal{L}$ in $G_3$ contains at most one edge of a color in $[n]$,
it has at least two edges of color $n+1$.
Thus $G_3$ contains no rainbow $\mathcal{L}$.
Hence, $ar(K^{(3)}_{n, n, n},\mathcal{L})\geq |C(G_3)|+1=n+2$.

For the upper bounds, let $G$ be an edge-colored $K^{(3)}_{n,n,n}$ with $|C(G)|\geq 3$ and $n\geq 3$.
Theorem~\ref{th:multi-tight} gives $|C(G)|\leq n+1$ if there is no rainbow $\mathcal{T}$;
Theorem~\ref{th:multi-messy} gives $|C(G)|\leq n$ if there is no rainbow $\mathcal{M}$;
Theorem~\ref{th:multi-loose} gives $|C(G)|\leq n+1$ if there is no rainbow $\mathcal{L}$.
This implies that $ar(K^{(3)}_{n, n, n},\mathcal{T})\leq n+2$, $ar(K^{(3)}_{n, n, n},\mathcal{M})\leq n+1$ and $ar(K^{(3)}_{n, n, n},\mathcal{L})\leq n+2$.
Therefore, we have $ar(K^{(3)}_{n, n, n},\mathcal{T})=ar(K^{(3)}_{n, n, n},\mathcal{L})=n+2$ and $ar(K^{(3)}_{n, n, n},\mathcal{M})=n+1$.
\hfill$\square$

\begin{remark}\label{re:multipar}
{\rm
One can check that the above proofs of Theorems~\ref{th:multi-tight}, \ref{th:multi-messy} and \ref{th:multi-loose} still work for edge-colorings of $K^{(3)}_{n_1, n_2, n_3}$ with $n_3\geq n_2\geq n_1\geq 3$.
Accordingly, Theorem~\ref{th:anti-Ram-multipar} can be generalized as follows: for integers $n_3\geq n_2\geq n_1\geq 3$, we have
$ar(K^{(3)}_{n_1, n_2, n_3},\mathcal{T})=\max\{n_3+1, n_1+2\}$,
$ar(K^{(3)}_{n_1, n_2, n_3},\mathcal{M})=n_3+1$ and
$ar(K^{(3)}_{n_1, n_2, n_3},\mathcal{L})=n_3+2$.
}
\end{remark}

\section{Concluding remarks}
\label{sec:conclu}

In this paper, we characterize the structures of rainbow $\mathcal{T}$-free, rainbow $\mathcal{M}$-free and rainbow $\mathcal{L}$-free edge-colorings of $K_n^{(3)}$, respectively.
As applications, we obtain several results related to constrained Ramsey numbers and anti-Ramsey numbers.
In particular, we prove $f(H, G)=R_2(H)$ for $G\in \{\mathcal{T}, \mathcal{M}, \mathcal{L}\}$ and infinitely many 3-uniform hypergraphs $H$.
In light of our Theorems~\ref{th:Ram-tight} and \ref{th:Ram-loose}, we pose the following two problems for further research.

\begin{problem}\label{prob:constrain-tight}
Is $f(H, \mathcal{T})=R_2(H)$ for every disconnected 3-graph $H$ with $R_2(H)\geq 5$?
\end{problem}

\begin{problem}\label{prob:constrain-loose}
Is $f(H, \mathcal{L})=R_2(H)$ for every 3-graph $H$ with $R_2(H)=|V(H)|\geq 7$?
\end{problem}

Our main results focus on 3-uniform paths of length three.
For higher uniformities or longer paths, several non-equivalent notions of hypergraph paths arise.
Even for a tight or loose path $P$, the structures of rainbow $P$-free edge-colorings become considerably more complicated.
Therefore, we propose the following problem as a first step toward higher uniformities or longer paths.

\begin{problem}\label{prob:long} Let $n$ be a positive integer.
\begin{itemize}
\item[{\rm (i)}] For the 3-uniform tight path $\mathbb{P}^{(3)}_{4}$ or loose path $P^{(3)}_4$, characterize rainbow $\mathbb{P}^{(3)}_{4}$-free and rainbow $P^{(3)}_4$-free edge-colorings of $K^{(3)}_n$, respectively.

\item[{\rm (ii)}] For $r\geq 4$, characterize rainbow $\mathbb{P}^{(r)}_{3}$-free and rainbow $P^{(r)}_3$-free edge-colorings of $K^{(r)}_n$, respectively.
\end{itemize}
\end{problem}

For general $r$-graphs $H$ and $G$, the constrained Ramsey number $f(H,G)$ need not exist.
This existence problem is closely related to the Canonical Ramsey Theorem of Erd\H{o}s and Rado~\cite{ErRa}.
Let $H$ be an $r$-graph with an ordered vertex set $\{v_1, v_2, \ldots, v_t\}$, i.e., $v_1<v_2<\cdots <v_t$.
Given a set $J\subseteq [r]$ and an edge $e=\{v_{i_1}, v_{i_2}, \ldots, v_{i_r}\}\in E(H)$ with $i_1<i_2<\cdots <i_r$, let $e : J \colonequals \{v_{i_j}: j\in J\}$.
We say an edge-coloring of $H$ is {\it ordered $J$-canonical} if for all edges $e,e'\in E(H)$, we have $c(e)=c(e')$ if and only if $e : J=e' : J$.
This notion is different from the $J$-canonical coloring in Definition~\ref{def:canonical-multipar}.
Note that for a fixed $J\subseteq [r]$, after renumbering the colors, all ordered $J$-canonical edge-colorings of $H$ are the same.

\begin{theorem}[Erd\H{o}s-Rado Canonical Ramsey Theorem~\cite{ErRa}]\label{th:canonical}
For any integers $r\geq 2$ and $t\geq 1$, there exists an integer $n$ such that in every edge-coloring of $K^{(r)}_{n}$, there is an ordered $J$-canonical copy of $K^{(r)}_{t}$ for some ordering of $V(K^{(r)}_{t})$ and some subset $J\subseteq [r]$.
\end{theorem}

\begin{proposition}\label{prop:constrained}
Let $H$ and $G$ be $r$-graphs.
Then the constrained Ramsey number $f(H,G)$ exists if and only if there exists a positive integer $t$ such that for every subset $J\subseteq [r]$,
$H$ occurs as a monochromatic subgraph of an ordered $J$-canonical edge-coloring of $K^{(r)}_{t}$
or $G$ occurs as a rainbow subgraph of an ordered $J$-canonical edge-coloring of $K^{(r)}_{t}$.
\end{proposition}

\begin{proof}
If $H$ and $G$ satisfy the condition in the proposition, then by Theorem~\ref{th:canonical}, there exists an integer $n$ such that every edge-coloring of $K^{(r)}_{n}$ contains an ordered $J$-canonical edge-colored $K^{(r)}_{t}$ for some $J\subseteq [r]$, and thus contains either a monochromatic copy of $H$ or a rainbow copy of $G$.
This implies that $f(H,G)$ exists.
Now suppose that for all positive integers $t$, there exists some subset $J_t\subseteq [r]$ such that
$H$ does not occur as a monochromatic subgraph of any ordered $J_t$-canonical edge-coloring of $K^{(r)}_{t}$,
and $G$ does not occur as a rainbow subgraph of any ordered $J_t$-canonical edge-coloring of $K^{(r)}_{t}$.
Then for every positive integer $n$, we can take an ordered $J_n$-canonical edge-coloring of $K^{(r)}_{n}$ that contains neither a monochromatic copy of $H$ nor a rainbow copy of $G$.
This implies that $f(H,G)$ does not exist.
\end{proof}

For $r$-graphs $H$ and $G$ that do not satisfy the condition in Proposition~\ref{prop:constrained}, we may instead study the following problem.

\begin{problem}\label{prob:Gallai}
Let $H$ and $G$ be two $r$-graphs and let $k$ be a positive integer.
\begin{itemize}
\item[{\rm (i)}] Determine the minimum integer $N$ such that for every integer $n\geq N$, every edge-coloring of $K^{(r)}_{n}$ with \textbf{at most} $k$ colors contains either a monochromatic copy of $H$ or a rainbow copy of $G$.

\item[{\rm (ii)}] Determine the minimum integer $N$ with ${N\choose r}\geq k$ such that for every integer $n\geq N$, every edge-coloring of $K^{(r)}_{n}$ with \textbf{exactly} $k$ colors contains either a monochromatic copy of $H$ or a rainbow copy of $G$.
\end{itemize}
\end{problem}

For $r$-partite $r$-graphs $H$ and $G$ that do not satisfy the condition in Proposition~\ref{prop:constrained-multipar}, we pose the multipartite analogue of Problem~\ref{prob:Gallai}. 
Such problems for $2$-uniform paths (i.e., ordinary graph paths) were studied in~\cite{LiLiu,LiWL}.

\begin{problem}\label{prob:Gallai-multipar}
Let $H$ and $G$ be two $r$-partite $r$-graphs and let $k$ be a positive integer.
\begin{itemize}
\item[{\rm (i)}] Determine the minimum integer $N$ such that for every integer $n\geq N$, every edge-coloring of $K^{(r)}_{n,\ldots, n}$ with \textbf{at most} $k$ colors contains either a monochromatic copy of $H$ or a rainbow copy of $G$.

\item[{\rm (ii)}] Determine the minimum integer $N$ with $N^r\geq k$ such that for every integer $n\geq N$, every edge-coloring of $K^{(r)}_{n,\ldots, n}$ with \textbf{exactly} $k$ colors contains either a monochromatic copy of $H$ or a rainbow copy of $G$.
\end{itemize}
\end{problem}

\vspace{-0.2cm}
\section*{Acknowledgements}
\vspace{-0.3cm}

This work was supported by the National Natural Science Foundation of China (Grant No. 12501492),
Shaanxi Fundamental Science Research Project for Mathematics and Physics (Grant No. 25JSQ043),
and the Fundamental Research Funds for the Central Universities (Grant No. GK202506024).

\vspace{-0.4cm}
\section*{Declaration of competing interests}
\vspace{-0.4cm}

The authors declare no competing interests.
\vspace{-0.4cm}

\section*{Data availability}
\vspace{-0.4cm}

No data were used for the research described in the article.

\begin{spacing}{0.8} 

\end{spacing}

\end{document}